\documentclass[letterpaper]{article}

\usepackage[english]{babel}
\usepackage{csquotes}
\usepackage{amsmath,amssymb,amsthm,mathtools}
\usepackage[shortlabels]{enumitem}
\usepackage{microtype}
\usepackage[dvipsnames,x11names]{xcolor}
\usepackage[hypertexnames=false]{hyperref}
\usepackage{modernfrontmatter}

\definecolor{HmDavide}{RGB}{139,0,0}

\hypersetup{
    colorlinks,
    linkcolor={red!50!black},
    citecolor={blue!50!black},
    urlcolor={MJAccent}
}

\numberwithin{equation}{section}

\newtheorem{theorem}{Theorem}[section]
\newtheorem{proposition}[theorem]{Proposition}
\newtheorem{lemma}[theorem]{Lemma}
\newtheorem{corollary}[theorem]{Corollary}

\theoremstyle{remark}
\newtheorem{remark}[theorem]{Remark}

\theoremstyle{definition}
\newtheorem{definition}[theorem]{Definition}
\newtheorem{example}[theorem]{Example}

\theoremstyle{plain}
\newtheorem{mainthm}{Theorem}

\newcommand{\R}{\mathbb R}
\newcommand{\N}{\mathbb N}
\newcommand{\Om}{\Omega}
\newcommand{\one}{\mathbf 1}
\newcommand{\dd}{\,\mathrm d}
\newcommand{\ellinfty}{\ell^\infty(X)}
\newcommand{\LG}{\hyperref[eq:LG]{\textup{(LG)}}}
\newcommand{\MB}{\hyperref[eq:nonlinear-generalized-mass]{\textup{(MB)}}}
\newcommand{\SC}{\hyperref[def:SC]{\textup{(SC)}}}
\newcommand{\SCinf}{\hyperref[def:SC-infinity]{\textup{{(SC\textsubscript{$\infty$})}}}}
\DeclareMathOperator{\dom}{dom}
\newcommand{\id}{\operatorname{id}}

\usepackage[style=numeric,backend=biber,giveninits=true,maxbibnames=99,sortcites=true]{biblatex}
\allowdisplaybreaks

\title{\texorpdfstring{Nonlinear parabolic characterizations of stochastic completeness at infinity on weighted graphs}{Nonlinear parabolic characterizations of stochastic completeness at infinity on weighted graphs}}
\articlelabel{}
\paperstatus{Preprint}
\runningtitle{Nonlinear parabolic characterizations of stochastic completeness at infinity on graphs}
\runningauthors{Bianchi \textperiodcentered\ Hua \textperiodcentered\ Setti \textperiodcentered\ Wojciechowski}
\date{}

\paperauthor[1]{Davide Bianchi}
\paperauthor[2]{Bobo Hua}
\paperauthor[3]{Alberto G. Setti}
\paperauthor[4,5]{Rados{\l}aw K. Wojciechowski}

\paperaffiliation[1]{School of Mathematics (Zhuhai), Sun Yat-sen University, Zhuhai 518055, China}
\paperaffiliation[2]{School of Mathematical Sciences, LMNS, Fudan University, Shanghai 200433, China}
\paperaffiliation[3]{Dipartimento di Scienze e Alta Tecnologia, Universit\`a dell'Insubria, Como 22100, Italy}
\paperaffiliation[4]{Department of Mathematics and Computer Science, York College -- CUNY, Jamaica, USA}
\paperaffiliation[5]{Department of Mathematics, Graduate Center -- CUNY, New York, USA}

\authoremail[D.B.]{bianchid@mail.sysu.edu.cn}
\authoremail[B.H.]{bobohua@fudan.edu.cn}
\authoremail[A.G.S.]{alberto.setti@uninsubria.it}
\authoremail[R.K.W.]{rwojciechowski@gc.cuny.edu}


\begin{document}
\maketitle
\begin{abstract}
We prove a nonlinear parabolic characterization of stochastic completeness at infinity for weighted graphs.  For the filtration equation
\[
    (\partial_t  + \Delta \Phi)u =0
\]
where $\Delta$ is the non-negative formal graph Laplacian and $\Phi u =\phi \circ u$ with $\phi \colon \R\to\R$ nonconstant, continuous and increasing, stochastic completeness at infinity
is equivalent to uniqueness of bounded pointwise solutions for every bounded initial datum. For \(\Phi=\id\), this recovers the classical heat equation
characterization of stochastic completeness at infinity, and of
stochastic completeness when the killing term is trivial, i.e., when \(\kappa=0\).
If stochastic completeness at infinity fails, then every bounded initial datum admits infinitely many bounded pointwise solutions of the filtration equation.  Admissible nonlinearities include the signed porous medium and fast diffusion powers $\phi(s)=s|s|^{m-1}$ for all $m>0$, as well as many others.

Stochastic completeness at infinity is further characterized by a generalized mass balance: the total mass of a positive pointwise solution at time $t$, augmented by the mass $\int_0^t\sum_{x}\kappa(x)\phi(u(s,x)) \dd s$ dissipated by the killing term $\kappa$, equals the initial mass.  This balance holds for every bounded positive solution on graphs of finite measure and for bounded finite-mass data on graphs of arbitrary measure under the sharp condition $\limsup_{r\to0^+}\phi(r)/r<\infty$.  It also extends to positive pointwise solutions
in $\ell^1$ that are bounded on every positive time interval.
When the killing term is trivial, stochastic completeness at infinity reduces to stochastic completeness and generalized balance to conservation of mass.
\end{abstract}

\section{Introduction}

Let $(M, g)$ be a complete, noncompact Riemannian manifold without boundary. The natural diffusion process on $M$ is Brownian motion, a continuous Markov process $\{X_t\}_{t \ge 0}$ generated by the Laplace--Beltrami operator $\Delta_M:=-\mathrm{div}(\mathrm{grad}).$ A fundamental question is whether the process $X_t$ remains in $M$ for all time or escapes to the ``ideal boundary'' at infinity in finite time. Formally, let $\zeta$ denote the lifetime (or explosion time) of the Brownian motion. A manifold $M$ is called stochastically complete if $\mathbb{P}_x(\zeta = \infty) = 1$ for every starting point $x \in M$. Analytically, this probabilistic property translates directly to the conservation of mass for the heat semigroup, i.e.,
$$e^{-t\Delta_M}1 (x)=\int_M p(t, x, y) \, \dd\mathrm{vol}(y)= 1 \quad \text{for all }  t>0 \text{ and } x\in M$$
where $p(t, x, y)$ denotes the minimal heat kernel, namely, the transition probability density of Brownian motion. With a slight abuse of notation, we use $1$ to denote both the scalar and the constant function equal to $1$.

The theory of stochastic completeness on Riemannian manifolds is an active area of research~\cite{GIM,masamune2020generalized,KuwaeLi2022,BisterzoMarini2023,GangulyPinchoverRoychowdhury2024,GIMP} that integrates probabilistic, analytic, and geometric perspectives, see, e.g., \cite{Gaffney1959,Grigoryan1986,Davies1992,KarpLi,Grigoryan1999,Grigoryan2009}.  In particular, its relationship with weak maximum principles at infinity was established and systematically developed by Pigola, Rigoli, and Setti~\cite{PigolaRigoliSetti2003,PigolaRigoliSetti2005}.

For a Riemannian manifold with a potential $V \colon M\to[0,\infty)$, Masamune and Schmidt~\cite{masamune2020generalized} developed a generalized conservation property for the heat semigroup generated by the Schr\"odinger operator, more specifically,
$$e^{-t(\Delta_M+V)}1+\int_0^te^{-s(\Delta_M+V)}V\dd s=1 \quad \text{for all }  t>0.$$
This definition was motivated by work done in the graph setting by Keller and Lenz \cite{KL12},
where this phenomenon is called stochastic completeness at infinity and
which distinguishes loss of heat through killing (or potential) from escape to infinity.

While stochastic completeness is classically formulated for the linear heat equation, nonlinear characterizations on Riemannian manifolds were established by Grillo, Ishige, and Muratori~\cite{GIM} and, more recently, by Grillo, Ishige, Muratori, and Punzo~\cite{GIMP}.
The former proved that stochastic completeness is equivalent both to uniqueness of bounded positive solutions of filtration equations with concave fast diffusion-type nonlinearities and to a corresponding semilinear Liouville property.  The latter removed the concavity, strict-monotonicity, and sign restrictions: for every continuous, increasing, nonconstant nonlinearity, every bounded initial datum, and every time horizon, stochastic incompleteness is equivalent to the existence of infinitely many bounded very weak solutions~\cite[Theorem~1.7]{GIMP}.  The corresponding elliptic characterizations are given in~\cite[Theorems~1.1 and~1.4]{GIMP}.

In the graph setting, stochastic completeness has been extensively studied since the works~\cite{Wojciechowski2008,Weber2010}; see also the survey~\cite{Wojciechowski2021Survey}
and~\cite{AS23, KMW25} for more recent developments.
The extension to stochastic completeness at infinity was carried out in \cite{KL12}; see also the book \cite{KLW21}.
We now recall the setting of weighted graphs. Let $G=(X,w,\kappa,\mu)$ be a weighted graph in the standard framework of discrete Dirichlet spaces~\cite{KLW21}, where $X$ is a countable set of vertices; $\mu \colon X\to(0,\infty)$ is the vertex measure; the edge weight $w \colon X\times X\to[0,\infty)$ is symmetric, $w(x,x)=0$, and satisfies
\[
    \sum_{y\in X} w(x,y)<\infty\qquad\text{for every }x\in X;
\]
and $\kappa\colon X\to[0,\infty)$ is the killing term. The formal
Laplacian is defined for $f \colon X\to \R$ and $x \in X$ by
\[
    \Delta f(x)=\frac{1}{\mu(x)}\sum_{y\in X}w(x,y)(f(x)-f(y))+\frac{\kappa(x)}{\mu(x)}f(x)
\]
where $f$ is chosen so that the sum is absolutely convergent. Note that the killing term plays the role of the potential of a Schr\"odinger operator.
Let $L$ denote the positive self-adjoint Markovian realization of the formal Laplacian $\Delta$ in $\ell^2(X,\mu)$ whose semigroup $e^{-tL}$ generates the minimal positive solution of the heat
equation.

We introduce the definition of stochastic completeness at infinity as introduced by \cite{KL12}.
\begin{definition}[Stochastic completeness at infinity]\label{def:SC-infinity}
A graph $G=(X,w,\kappa,\mu)$ satisfies \emph{stochastic completeness at infinity} \SCinf{} if for some (and hence all) $t>0,$
$$e^{-tL}1+\int_0^t e^{-sL}\left(\frac{\kappa}{\mu}\right)\dd s=1.$$
When $\kappa=0$ and the above holds, then $G$ is said to  satisfy \emph{stochastic completeness} \SC{}.\label{def:SC}
Otherwise, a graph is said to be \emph{stochastically incomplete (at infinity)}.
\end{definition}

Note that the heat equation on graphs is related to the continuous-time random walk
and $G$ satisfies  \SCinf{}
when the minimal heat semigroup loses no
heat at infinity, so that any dissipated heat is accounted for entirely by
the killing term $\kappa.$
Equivalent resolvent, bounded elliptic, heat equation, and
Omori--Yau characterizations are collected in
Proposition~\ref{prop:SCinf-characterizations} below.

The main contributions of \cite{GIM,GIMP} consist of nonlinear elliptic and parabolic characterizations of stochastic completeness on Riemannian manifolds.
On graphs, Schmidt and Zimmermann~\cite{SZ} established a discrete analog of the nonlinear elliptic characterization from \cite{GIMP}, as recalled below in Theorem~\ref{thm:S-Z-liouville},
and mentioned a nonlinear parabolic characterization as an open problem.
In this paper, we prove such a nonlinear parabolic characterization of \SCinf{} on graphs.

We consider the class of nonlinearities
\[
\mathcal I
:=
\left\{
\phi\colon\R\to\R
\;\middle|\;
\phi\text{ is continuous, increasing, and nonconstant, with }\phi(0)=0
\right\}
\]
where increasing is always understood in the weak sense, that is,
\(\phi(s_1)\leq\phi(s_2)\) whenever \(s_1 < s_2\). For any $u \colon X\to\R$ and $\phi\in\mathcal{I},$ we set
\[
    \Phi u:=\phi\circ u.
\]

The Cauchy problem of the generalized porous medium equation (GPME), or filtration equation,
for an initial datum $u_0 \colon X\to\R,$
is given as follows
\begin{equation}\tag{GPME}\label{eq:GPME}
	\begin{cases}
		(\partial_t +\Delta\Phi)u(t,x) = 0
		& \text{for all } (t,x) \in (0,T)\times X,\\
		u|_{t=0} = u_0.
		&
	\end{cases}
\end{equation}

The signed power example is
\[
    \phi(s)=s^m:=s|s|^{m-1}\qquad \text{for }m>0.
\]
Then $m=1$ gives the heat equation, $m>1$ the porous medium range, and $0<m<1$ the fast diffusion range; see, for instance, V\'azquez's monograph~\cite{vazquez2007porous} and the graph setting developed in~\cite{bianchi2022generalized,bianchi2026bounded}. The nonlinearity also includes the Stefan problems, i.e., $\phi(s)=(s-a)_+$ for $a>0$, see, e.g., \cite{vazquez2007porous}.

Over Euclidean
domains, subsequent work has concentrated on the regularity
of solutions, on their short- and long-time asymptotics, and on singular
and extinction phenomena for the associated porous medium and fast
diffusion
flows~\cite{jin2022singular,jin2026extinction,grillo2026rough,bonforte2010veryfast}.
These questions have also been studied in the Riemannian
setting, where smoothing effects, Aronson--B\'enilan type gradient bounds,
and the behavior of solutions are governed by the
underlying
geometry~\cite{grillo2016cartanhadamard,wang2021gradient,fujitani2024aronson,grillo2025green,bianchi2018laplacian}.
More recently, and closer in spirit to the present work, filtration and
porous medium-type equations have begun to be studied on metric measure
spaces and on discrete
structures~\cite{Ma2022,bianchi2022generalized, bianchi2026bounded,kou2026porous,berchio2026semilinear,berchio2026fractional}.

On graphs, nonlinear parabolic problems have been developed along several
complementary lines. The parabolic discrete $p$-Laplacian and
Leray--Lions evolutions were treated within nonlinear semigroup
theory~\cite{Mugnolo2013DiscretePLaplacian,HuaMugnolo15,MazonSoleraToledo2020LerayLions},
while a variational framework for nonlinear elliptic equations, including
Yamabe- and Kazdan--Warner-type problems, was introduced in~\cite{GrigoryanLinYang2016Yamabe,GrigoryanLinYang2016KazdanWarner}.
For semilinear reaction--diffusion
equations $(\partial_t  + \Delta) u = f(u)$, blow-up, global existence,
and Fujita-type critical phenomena have been established for power and
convex
sources~\cite{LinWu2017Semilinear,Wu2021Fujita,LenzSchmidtZimmermann2023}, alongside more recent existence, lifespan, and nonexistence results for
general, time-dependent, or dissipative sources on infinite
graphs~\cite{GrilloMeglioliPunzo2026,PunzoSacco2026,MonticelliPunzoSomaglia2026,berchio2026semilinear}.

We next state one of our main results, which follows from Theorems~\ref{thm:killing-unique}~and~\ref{thm:killing-nonlinear-nonunique}. This provides a nonlinear parabolic characterization of \SCinf{} on graphs, extending the result in \cite{GIMP} to the discrete setting and even including the killing term.
We note that, by results found in \cite{bianchi2026bounded}, there
always exists a bounded pointwise solution to the filtration equation
given a bounded initial datum. The result below
says that stochastic completeness at infinity is equivalent to the uniqueness
of this solution.
Furthermore, we note that, in the heat equation case, i.e., when $\phi(s)=s$,
this is one of the known formulations of stochastic completeness. Hence, this result says
that uniqueness
of bounded solutions for $\phi(s)=s$ implies uniqueness for all $\phi \in \mathcal{I}$.
\begin{mainthm}[see Theorem~\ref{thm:killing}]\label{thm:mainA}
A graph \(G\) satisfies \SCinf{} if and only if the \ref{eq:GPME} has a unique bounded pointwise solution, i.e., for some (equivalently, all) $\phi\in\mathcal{I}$ and some (equivalently, all) $u_0\in\ell^\infty(X),$ any two bounded pointwise solutions of the \ref{eq:GPME} with initial datum $u_0$ coincide.
\end{mainthm}

Furthermore, we prove that a weighted graph does not satisfy \SCinf{} if and only if there are infinitely many bounded solutions, see Theorem~\ref{thm:killing-nonlinear-nonunique}. The equivalent conditions for \SCinf{} established in this paper, including an additional comparison principle characterization, are collected in Theorem~\ref{thm:killing}. The proof combines graph-specific exhaustion and comparison arguments with ideas parallel to those used in the manifold setting.  In particular, the exterior-Dirichlet construction and the concave modulus of continuity are discrete counterparts of the arguments in~\cite[proof of Theorem~1.7 and Lemma~5.1]{GIMP}.  The additional issues here arise from the lack of a chain rule, the pointwise solution framework, and the presence of a killing term. Moreover, we avoid any local finiteness assumption, that is, we do not assume that $|\{y\in X \mid w(x,y)>0\}|<\infty$ for all $x\in X.$  The weak Omori--Yau characterization used in the proof comes from~\cite[Corollary~7.29]{KLW21}, for its original no-killing graph version, see~\cite{Huang2011OmoriYau}.  The modulus estimate is essential in the fast diffusion range, where $\phi(s)=s|s|^{m-1}$ is not locally Lipschitz at the origin for $0<m<1$.

We complement the uniqueness dichotomy with a mass characterization.
We first define the generalized mass balance property for solutions as follows.
\begin{definition}[Generalized mass balance]
\label{def:generalized-mass-conservation}
Let $T>0$ and let $u \colon [0,T]\times X\to[0,\infty)$ be a pointwise solution of the \ref{eq:GPME} with initial datum $u_0$.
We say that $u$ satisfies the \emph{generalized mass balance} on $[0,T]$ if
\begin{equation}\tag{MB}\label{eq:nonlinear-generalized-mass}
     \sum_{x\in X}u(t,x)\mu(x)
    +\int_0^t\sum_{x\in X}\kappa(x)\,\phi(u(s,x))\dd s
    =\sum_{x\in X}u_0(x)\mu(x)
    \qquad \text{for all } 0\le t\le T.
\end{equation}
\end{definition}
In particular, when \MB{} holds, this means that
the total mass dissipated by the killing term is measured by the
nonlinear flux $\Phi u$.

We define a one-sided linear growth condition at the origin for the nonlinearity as follows
\begin{equation}\tag{LG}\label{eq:LG}
    \limsup_{r\downarrow0}\frac{\phi(r)}{r}<\infty.
\end{equation}
The next main result is that under some conditions, stochastic
completeness at infinity holds if and only if every bounded positive
solution of the \ref{eq:GPME} with finite-mass initial
data satisfies the generalized mass balance. Let $\R_+:=(0,\infty).$
\begin{mainthm}[see Theorem~\ref{thm:generalized-mass-characterization}] Let $G$ be a graph and $\phi\in\mathcal{I}$ with $\phi^{-1}(\R_+)\neq\emptyset.$ Assume that either
   \[ \mu(X)<\infty
    \qquad\text{or}\qquad
    \phi\text{ satisfies }\LG{}.\]
Then \(G\) satisfies \SCinf{} if and only if
for every \(T>0\) and every bounded
    positive pointwise solution \(u\) of the \ref{eq:GPME} with initial
    datum \(u_0\in \ell^1(X,\mu) \cap \ell^\infty(X)\), $u$ satisfies the generalized mass
    balance~\MB{}.
\end{mainthm}

On graphs of finite measure this covers every
bounded datum and every nonlinearity; on graphs of arbitrary measure it
holds under the sharp linear growth condition
\LG{}, which admits the linear and
porous medium powers.  The balance also extends to positive
$\ell^1$-valued pointwise solutions that are bounded on every positive-time
interval.  No summability assumption on
$\kappa$ is imposed: stochastic completeness at infinity itself supplies
the required integrability through a no-flux identity at infinity, see
Lemma~\ref{lem:no-flux-killing}.  The growth and boundedness hypotheses
are shown to be necessary by explicit birth--death chain examples, see
Example~\ref{ex:sharpness-mass-extensions}.  In
the no-killing case, the balance reduces to conservation of mass and
stochastic incompleteness is then detected by the simultaneous loss and
creation of mass from the same constant initial datum, see
Corollary~\ref{cor:mass-no-killing}.

Note that none of our main results require the underlying graph to be locally finite.
Additionally, our framework naturally allows for killing terms.
This extends our findings to nonlocal and Schrödinger operators with positive potentials.

The paper is organized as follows.  Section~\ref{sec:setting} fixes the setting and recalls
the notion of stochastic completeness, its version at infinity, and
various equivalent formulations.
Section~\ref{sec:comparison} collects the finite subgraph comparison principles and the construction of extremal bounded solutions from~\cite{bianchi2026bounded}.  Section~\ref{sec:unique-infinity} proves uniqueness under stochastic completeness at infinity, Section~\ref{sec:nonunique-infinity} constructs infinitely many solutions when it fails, and Section~\ref{sec:parabolic-characterization} combines the two results into a characterization, which also includes detection by a single initial datum and a global comparison principle. A generalized mass-conservation formulation for pointwise solutions, including an extension to $\ell^1$-valued solutions bounded away from the initial time, is given in Section~\ref{sec:mass-characterization},
together with a measure change characterization that reduces stochastic completeness at infinity to ordinary stochastic completeness.

\section{Notation and background}\label{sec:setting}
In this section, we introduce the setting and notations.
We then recall several equivalent formulations of stochastic completeness at infinity
and mention, for context and completeness, the recent work of Schmidt--Zimmermann~\cite{SZ}.

Throughout, $X$ is countable, $\mu \colon X\to(0,\infty)$, $w \colon X\times X\to[0,\infty)$ is symmetric with $w(x,x)=0$ and
\[
    \sum_{y\in X} w(x,y)<\infty\qquad\text{for every }x\in X.
\]
The killing term is a function $\kappa \colon X\to[0,\infty)$; we refer to $\kappa=0$ as the
\emph{no-killing case}. We call the quadruple $G=(X, w, \kappa, \mu)$ a \emph{graph}.

In this paper, we always consider \emph{connected} graphs, i.e.,
for any $x,y\in X,$ there exists a path
$(x_i)_{i=0}^N$ such that $w(x_i,x_{i+1})>0$ for all $0\leq i\leq N-1,$ $x_0=x$, and $x_N=y.$
If $\Om \subseteq X$, then $\Om$ induces a subgraph by restricting $w$ to $\Om \times \Om$
and $\mu$ and $\kappa$ to $\Om$. We then call the subset \emph{connected} if the graph
induced by the subset is connected.
We call a sequence of finite connected subsets $\Om_n \subseteq X$
such that $\Om_n \subseteq \Om_{n+1}$ for all $n$ and $X= \bigcup_n \Om_n$
an \emph{exhaustion} of $G$.

We introduce the usual $\ell^p$ spaces by
\begin{align*}
\ell^p(X,\mu)&= \{ f \colon X\to\R \mid \sum_{x \in X} |f(x)|^p \mu(x)< \infty \}	&\mbox{for } p \in [1,\infty)\\
\ell^\infty(X) &= \{ f  \colon X\to\R  \mid \sup_{x \in X}
|f(x)| < \infty\} &\mbox{for } p=\infty
\end{align*}
	with norms
\[
\|f\|_p= \begin{cases}
	\left(\sum_{x\in X} |f(x)|^p\mu(x)\right)^{{1}/{p}} & \mbox{for } p \in [1,\infty)\\
	\sup_{x\in X}|f(x)| & \mbox{for } p =\infty.
\end{cases}
\]

For a function $f\colon X \to \R$ we write $f \geq 0$ if $f(x)\geq0$ for all $x \in X$
and call such functions \emph{positive}. A function such that $-f$ is positive is
called \emph{negative}.

For an interval $I \subseteq \R$, we denote by
$C(I;\ell^p(X,\mu))$ and $C^1(I;\ell^p(X,\mu))$
the set of continuous, respectively, continuously differentiable, functions from the interval
to the $\ell^p$-space
with respect to the $\ell^p$-norm.
When $\ell^p(X,\mu)$ is replaced by $\R$ in the above,
we simply omit it in the notations.

Throughout, $\phi\in\mathcal{I}$ where
$$\mathcal I:=\{\phi\colon\R\to\R \mid \phi\text{ is continuous, increasing, nonconstant, with }\phi(0)=0\}$$
and $\Phi u:=\phi\circ u$.

Let
\[
    \mathcal F:=\left\{f\colon X\to\R \mid \sum_{y\in X}w(x,y)|f(y)|<\infty\text{ for every }x\in X\right\}.
\]
For $f\in\mathcal F$ and $x \in X$, we define the Laplacian $\Delta$ via
\[
    \Delta f(x)=\frac{1}{\mu(x)}\sum_{y\in X}w(x,y)(f(x)-f(y))+\frac{\kappa(x)}{\mu(x)}f(x).
\]
Since $\sum_{y\in X}w(x,y)<\infty$ for every $x\in X$, every bounded function belongs to $\mathcal F$, that is, $\ellinfty\subseteq \mathcal F$.

For the nonlinearity $\phi \in \mathcal{I}$, set
\[
    \mathcal{F}_\Phi:=\{f \colon X\to\R \mid \Phi f\in\mathcal F\}
\]
and, for $f \in \mathcal{F}_\Phi$ and $x \in X$, let
 $$\Delta\Phi f:=\Delta(\phi\circ f).$$
If $f$ is bounded, then so is $\Phi f$, hence, $\ellinfty\subseteq \mathcal{F}_\Phi$.

We now define the notion of pointwise solutions of the \ref{eq:GPME}.
\begin{definition}[Pointwise solution]\label{def:pointwise}
Let $T>0$ and $u_0 \colon X\to\R$.  A \emph{pointwise solution of the \ref{eq:GPME}
on $[0,T]\times X$ with initial datum $u_0$}
is a function $u \colon [0,T]\times X\to\R$ such that
\begin{enumerate}[label=(\alph*)]
    \item[\textup{(a)}] $u(t,\cdot)\in\mathcal{F}_\Phi$ for every $t\in(0,T)$;
    \item[\textup{(b)}] $u(\cdot,x)\in C([0,T])\cap C^1((0,T))$ for every $x\in X$;
    \item[\textup{(c)}] $(\partial_t +\Delta\Phi)u(t,x)=0$ for every $(t,x)\in(0,T)\times X$;
    \item[\textup{(d)}] $u(0,x)=u_0(x)$ for every $x\in X$.
\end{enumerate}
The solution $u$ is a \emph{bounded pointwise solution} if $u$ is bounded.
\end{definition}

We note that for bounded
solutions, condition~(a) is automatically satisfied because
$\ellinfty\subseteq\mathcal{F}_\Phi$.

A function $u \colon [0,\infty)\times X\to\R$ is a \emph{global pointwise
solution} if the restriction of $u$ to $[0,T]\times X$ satisfies conditions
\textup{(a)}--\textup{(d)} above for every $T>0$.  It is a \emph{global
bounded pointwise solution} if, in addition, it is bounded on
$[0,\infty)\times X$.

The following semilinear Liouville theorem of Schmidt--Zimmermann~\cite{SZ} is the elliptic counterpart, in the no-killing case, of the parabolic characterization proved in this paper.  It is stated for context and is not used in the proofs below. A function $\phi\colon\R\to\R$ is \emph{strictly increasing} if
\(\phi(s_1)<\phi(s_2)\) whenever \(s_1 < s_2\). For $\psi\in\mathcal I$ and $h\colon X\to\R$, we write $\Psi h:=\psi\circ h$.

\begin{theorem}[Schmidt--Zimmermann]\label{thm:S-Z-liouville}
Assume $\kappa=0$.  A graph $G$ satisfies \SC{} if and only if for one (equivalently, all) strictly increasing $\psi\in\mathcal I$, the only $0\le h\in\ellinfty$ satisfying
\begin{equation}\label{eq:nonlinear-liouville-eq}
    -\Delta h=\Psi h
\end{equation}
pointwise on $X$ is $h=0$.  In particular, if $G$ does not satisfy \SC{}, then for every strictly increasing $\psi\in\mathcal I$ there exists a nontrivial $0\le h\in\ellinfty$ satisfying~\eqref{eq:nonlinear-liouville-eq}.
\end{theorem}

\begin{remark}
Theorem~\ref{thm:S-Z-liouville} directly restates
the equivalence (i) $\Longleftrightarrow$ (vi) of~\cite[Theorem~5.1]{SZ} together with~\cite[Corollary~5.2]{SZ}, specialized to the constant weight $W=1$. The formal Laplacian
$\mathcal L$ in \cite{SZ}
carries no killing, so it agrees with our $\Delta$ when $\kappa=0$, and their equation ``$-\mathcal Lf=\psi(Wf)$" reads $-\Delta h=\Psi h$.  Theorem~5.1 of~\cite{SZ} contains further equivalent formulations, in terms of super- and subsolutions and of a nonlinear resolvent.
\end{remark}

We next introduce some equivalent formulations of stochastic completeness at infinity
that will be useful in what follows.
For \(\lambda>0\) and \(h\in\mathcal{F}\), set
\[
    S_\lambda h:=(\id+\lambda\Delta) h.
\]
Functions $h \in \mathcal{F}$ such that $S_\lambda h=0$ are called \emph{$1/\lambda$-harmonic}.
Furthermore, we say that $G$ satisfies the \emph{weak Omori--Yau maximum
principle} if, for every $f\in\mathcal F$ bounded above with
$f^*:=\sup_X f\in(0,\infty)$, the super-level sets
$X_\delta:=\{x\in X \mid f(x)>f^*-\delta\}$ satisfy
\[
    \sup_{x\in X_\delta}\Delta f(x)\ge 0
    \qquad\text{for every }\delta\in(0,f^*).
\]
The following statement gathers the various formulations of stochastic completeness
at infinity that we will need in this paper.

\begin{proposition}[Known characterizations of \SCinf{}]
\label{prop:SCinf-characterizations}
Let $G$ be a graph.
The following are equivalent:
\begin{enumerate}[label=\textup{(\roman*)}]
\item \(G\) satisfies \SCinf{}.
\item For one (equivalently, all) \(\lambda>0\),
      \(S_\lambda=\id+\lambda\Delta\) is injective on \(\ellinfty\).
\item For one (equivalently, all) \(\beta>0\), every
      \(0\le h\in\ellinfty\) satisfying
      \((\Delta+\beta)h\le0\) vanishes.
\item Bounded solutions of the heat equation are uniquely determined
      by bounded initial data.
\item $G$ satisfies the weak Omori--Yau maximum principle.
\end{enumerate}
\end{proposition}

\begin{proof}[Sketch of proof]
The equivalence of \textup{(i)}--\textup{(iii)} follows from
\cite[Theorem~1]{KL12}, see also \cite[Theorem~7.18]{KLW21}.
Indeed, \(S_\lambda h=0\) is the same as
\((\Delta+\lambda^{-1})h=0\), and the cited results characterize failure
of \SCinf{} by the existence of a nontrivial bounded nonnegative
subsolution of \((\Delta+\beta)h=0\) for $\beta>0$.
The equivalence of \textup{(i)} and \textup{(iv)} can also be found
in \cite[Theorem~1]{KL12}, see also
\cite[Theorem~7.16]{KLW21}. Note that the existence of a bounded solution
to the heat equation is
always provided by the minimal heat semigroup, so this is a uniqueness
characterization.  Finally, the equivalence between \textup{(i)} and \textup{(v)} is
proven in~\cite[Corollary~7.29]{KLW21}, see \cite{Huang2011OmoriYau}
for the original result.
\end{proof}

\section{Finite subgraph comparison and extremal bounded solutions}\label{sec:comparison}
In this section, we summarize recent results from \cite{bianchi2026bounded}
that will be needed in what follows.
In particular, we recall some comparison principles as well as the
construction and basic properties of pointwise solutions of the \ref{eq:GPME}.

For $\Omega\subseteq X$, let
\[
    \partial_e\Omega
    :=\{y\in X\setminus\Omega \mid w(x,y)>0\text{ for some }x\in\Omega\}
\]
be the \emph{exterior boundary} of $\Om$.  Given a map $\theta \colon \R\to\R$, an exterior datum
$\eta \colon X\setminus\Omega\to\R$ is \emph{$(\theta,\Omega)$-admissible} if
\[
    \sum_{y\notin\Omega}w(x,y)|\theta(\eta(y))|<\infty
    \qquad\text{for every }x\in\Omega.
\]
If $\eta$ is prescribed only on $\partial_e\Omega$, we extend it by zero to
$X\setminus(\Omega\cup\partial_e\Omega)$.  Constant exterior data are
identified with the corresponding constant functions on $X\setminus\Omega$
and are always $(\theta,\Omega)$-admissible.

Let now $\Omega\subseteq X$ be finite and let
$\eta \colon X\setminus\Omega\to\R$ be $(\operatorname{id},\Omega)$-admissible.
For $u \colon \Omega\to\R$, write $E_\Omega^\eta u \colon X \to \R$ for the extension
of $u$ by $\eta$
which is equal to $u$
on $\Omega$ and to $\eta$ on $X\setminus\Omega$ and write $\pi_\Omega$ for
the restriction of functions on $X$ to $\Omega$.  Denote by $C(\Om)$ the set of all functions $f\colon\Om\to\R$.  The finite subgraph
Laplacian with exterior datum $\eta$, $\Delta_\Omega^\eta\colon C(\Om)\to C(\Om)$, is defined as
\[
    \Delta_\Omega^\eta u:=\pi_\Omega\,\Delta(E_\Omega^\eta u),
\]
that is, for $x\in\Omega$,
\[
    \Delta_\Omega^\eta u(x)
    =\frac1{\mu(x)}\sum_{y\in\Omega}w(x,y)(u(x)-u(y))
    +\frac1{\mu(x)}\sum_{y\notin\Omega}w(x,y)(u(x)-\eta(y))
    +\frac{\kappa(x)}{\mu(x)}u(x).
\]
For the nonlinear equation, if $\eta$ is $(\phi,\Omega)$-admissible, one
prescribes the exterior datum before applying $\phi$:
\[
    \Delta_\Omega^\eta\Phi u:=\pi_\Omega\,\Delta\Phi(E_\Omega^\eta u)
    =\Delta_\Omega^{\phi\circ\eta}(\phi\circ u).
\]

The following is the sub- and super-solution form of the finite subgraph
parabolic comparison principle in~\cite[Theorem~3.1 and
Corollary~3.2]{bianchi2026bounded}.

\begin{lemma}\label{lem:finite-comparison}
Let $\Omega\subseteq X$ be finite and let $\eta_1,\eta_2 \colon [0,T]\times(X\setminus\Omega)\to\R$ be exterior data such that
$\eta_i(t,\cdot)$ is $(\phi,\Omega)$-admissible for every $t\in[0,T]$ and
$i=1,2$.  Let $u_1,u_2 \colon [0,T]\times\Omega\to\R$ satisfy
\[
    u_i(\cdot,x)\in C([0,T])\cap C^1((0,T))
    \qquad \text{for } x\in\Omega,\ i=1,2
\]
and
\[
    \left(\partial_t +\Delta_\Omega^{\eta_1(t,\cdot)}\Phi\right) u_1
    \le
    \left(\partial_t +\Delta_\Omega^{\eta_2(t,\cdot)}\Phi\right) u_2
    \qquad\text{on }(0,T)\times\Omega .
\]
If $u_1(0,\cdot)\le u_2(0,\cdot)$ on $\Omega$ and $\eta_1\le\eta_2$ on $[0,T]\times\partial_e\Omega$, then
\[
    u_1\le u_2\qquad\text{on }[0,T]\times\Omega .
\]
\end{lemma}

We complement the above with a well-known elliptic comparison principle.

\begin{lemma}[{\cite[Theorem~1.7]{KLW21}}]\label{lem:finite-elliptic-comparison}
Let $\Omega\subseteq X$ be finite.  Let $f,g\colon \Omega\to\R$ and let
$\xi,\eta \colon X\setminus\Omega\to\R$ be
$(\operatorname{id},\Omega)$-admissible exterior data.  Suppose that
$\xi\le\eta$ on $\partial_e\Omega$.  If
\[
    \Delta_\Omega^\xi f\leq\Delta_\Omega^\eta g
    \qquad\text{on }\Omega,
\]
then $f\le g$ on $\Omega$.
\end{lemma}

The following result is a homogeneous specialization of
\cite[Theorem~3.4]{bianchi2026bounded}.  The construction there proceeds by a
finite-set exhaustion with constant exterior data, yields solutions
defined for all time and is valid for arbitrary killing $\kappa\ge0$.

\begin{theorem}\label{thm:extremal}
Let $u_0\in\ellinfty$.  Choose constants
\[
    A\le \min\{0,\inf_X u_0\}
    \qquad \text{and} \qquad
    B\ge \max\{0,\sup_X u_0\}.
\]
Then, there exist global bounded pointwise solutions $u^A,u^B$ of
the \ref{eq:GPME} with initial datum $u_0$ satisfying
\[
    A\le u^A\le u^B\le B
    \qquad\text{on }[0,\infty)\times X.
\]
For every $T>0$, the restriction of $u^A$ to $[0,T]\times X$ is minimal
among all bounded pointwise solutions on that time interval lying above $A$
while the restriction of $u^B$ is maximal among all bounded pointwise
solutions lying below $B$.  In particular, if $v$ is any bounded pointwise
solution on $[0,T]\times X$ with $A\le v\le B$, then
\begin{equation}\label{eq:trapping}
    u^A\le v\le u^B
    \qquad\text{on }[0,T]\times X.
\end{equation}
If $u_0\ge0$, the choice $A=0$ gives a global minimal positive solution $u^0$
satisfying
\[
    0\le u^0(t,x)\le \|u_0\|_\infty
    \qquad\text{on }[0,\infty)\times X.
\]
\end{theorem}

\section{Uniqueness from stochastic completeness at infinity}\label{sec:unique-infinity}
In this section, we prove half of our first characterization by showing
that stochastic completeness at infinity implies uniqueness of bounded
solutions of the filtration equation.

Our uniqueness argument uses a nonlinear modulus estimate rather than a Lipschitz estimate; this is what allows for fast diffusion nonlinearities, which are not locally Lipschitz at the origin.  We first establish the auxiliary lemmas that we need. For the next Lemma~\ref{lem:modulus}, see also~\cite[Lemma~5.1]{GIMP}.

\begin{lemma}\label{lem:modulus}
Let $I\subset\R$ be a compact interval of length $l$ and let $\phi \colon I\to\R$ be continuous and increasing.  Then there exists a continuous, strictly increasing, concave function $\omega \colon [0,l]\to[0,\infty)$ with $\omega(0)=0$ such that
\begin{equation}\label{eq:modulus-bound}
    0\le \phi(b)-\phi(a)\le\omega(b-a)
    \qquad\text{whenever }a,b\in I
    \quad \text{with } a\le b.
\end{equation}
Consequently, $\omega^{-1}$ is convex and strictly increasing on $[0,\omega(l)]$.
\end{lemma}

\begin{proof}
Let
\[
    m(r):=\sup\{\phi(b)-\phi(a) \mid a,b\in I,\ 0\le b-a\le r\}
    \qquad \text{for }0\le r\le l.
\]
Then $m(0)=0$, $m$ is increasing, and $m(r)\downarrow0$ as $r\downarrow0$ by the uniform continuity of $\phi$ on $I$.  Splitting an increment of length at most $r+s$ into two increments of lengths at most $r$ and $s$ shows that $m$ is subadditive; together with the previous properties this makes $m$ continuous on $[0,l]$.  Let $\hat\omega$ be the least concave majorant of $m$ on $[0,l]$, which is finite, increasing,
and concave.  Given $\varepsilon>0$, choose $\delta>0$ with $m\le\varepsilon$ on $[0,\delta]$, then the concave function $r\mapsto\varepsilon+m(l)\,r/\delta$ majorizes $m$, so $\hat\omega(0)\le\varepsilon$.  Hence, $\hat\omega(0)=0$.

We now show that $\hat\omega$ is continuous on $[0,l]$.
By the minimality of $\hat\omega$, the same affine majorant gives
\[
0\le \hat\omega(r)\le \varepsilon+\frac{m(l)}{\delta}r
\qquad \text{for all  } 0\le r\le l.
\]
Thus, $\hat\omega(r)\to0=\hat\omega(0)$ as $r\downarrow0$.  Finally, fixing $a<l$, concavity gives, for $a<r<l$,
\[
\hat\omega(r)\ge
\frac{l-r}{l-a}\hat\omega(a)+\frac{r-a}{l-a}\hat\omega(l),
\]
while monotonicity gives $\hat\omega(r)\le\hat\omega(l)$. Hence, $\hat\omega(r)\to\hat\omega(l)$ as $r\uparrow l$. Therefore, $\hat\omega$ is continuous on $[0,l]$.

Now, for any fixed $\varepsilon_0>0$, the function $\omega(r):=\hat\omega(r)+\varepsilon_0r$ is continuous, strictly increasing, concave, vanishes at $0$ and satisfies~\eqref{eq:modulus-bound}.  The final assertion follows from the elementary fact that the inverse of a continuous, strictly increasing, concave function vanishing at $0$ is convex and strictly increasing.
\end{proof}

The following lemma is the key tool for proving uniqueness of solutions under
stochastic completeness at infinity.
It will be used repeatedly in what follows.
\begin{lemma}\label{lem:averaged-modulus}
Assume that $G$ satisfies \SCinf{}.  Let $T,R>0$
and let $q,\rho \colon [0,T]\times X\to[0,\infty)$ be bounded functions such
that, for every $x\in X$,
\[
    q(\cdot,x),\rho(\cdot,x)\in C([0,T])
    \qquad \text{and} \qquad
    q(\cdot,x)\text{ is locally absolutely continuous on }(0,T).
\]
Suppose that
\[
    q(0,x)=0,
    \qquad
    \partial_tq(t,x)+\Delta\rho(t,x)\le0
    \quad\text{for every }x\in X\text{ and almost every }t\in(0,T)
\]
and that there exists a continuous, strictly increasing, concave function
$\omega \colon [0,R]\to[0,\infty)$ with $\omega(0)=0$ such that
\[
    0\le q\le R
    \qquad \text{and} \qquad
    0\le\rho\le\omega(q).
\]
Then, $q=\rho=0$ on $[0,T]\times X$.
\end{lemma}

\begin{proof}
Set
\[
    \eta(t):=\frac{e^{-t}}{1-e^{-T}}
    \qquad \text{for } 0\le t\le T
\]
and define
\[
    W(x):=\int_0^T\eta(t)\rho(t,x)\dd t
    \qquad \text{and} \qquad
    Q(x):=\int_0^T\eta(t)q(t,x)\dd t.
\]
Then, $\eta(t)\dd t$ is a probability measure on $[0,T]$ and $W$ and $Q$ are bounded
and positive with $0\le W\le\omega(R)$.  In particular,
$W\in\ellinfty\subseteq\mathcal F$.  Boundedness of $\rho$ and the
summability of $w(x,\cdot)$ permit the interchange of $\Delta$ with the
time integral. Hence,
$$\Delta W(x)=\int_0^T\eta(t)\Delta\rho(t,x)\dd t\quad \text{for }x\in X.$$

Fix $x\in X$ and $0<\varepsilon<T-\delta<T$ where $\delta \in (0,T)$.  Since
$q(\cdot,x)$ is absolutely continuous on
$[\varepsilon,T-\delta]$, multiplying the differential inequality by
$\eta$, integrating by parts, and using that $\partial_t \eta(t)=-\eta(t)$ gives
\[
\begin{aligned}
    \int_\varepsilon^{T-\delta}\eta(t)\Delta\rho(t,x)\dd t
    &\le
    -\int_\varepsilon^{T-\delta}\eta(t)\partial_tq(t,x)\dd t\\
    &=
    -\eta(T-\delta)q(T-\delta,x)
    +\eta(\varepsilon)q(\varepsilon,x)
    -\int_\varepsilon^{T-\delta}\eta(t)q(t,x)\dd t .
\end{aligned}
\]
For fixed $x$, the function $\Delta\rho(\cdot,x)$ is bounded on
$[0,T]$.  Letting first $\delta\downarrow0$ and then
$\varepsilon\downarrow0$, using the continuity of $q(\cdot,x)$ and
$q(0,x)=0$, yields
\[
    \Delta W(x)
    \le-\eta(T)q(T,x)-Q(x).
\]
Consequently,
\begin{equation}\label{eq:averaged-modulus-DW}
    -\Delta W\ge Q.
\end{equation}

Since $\omega^{-1}$ is convex and strictly increasing on
$[0,\omega(R)]$, the modulus bound and Jensen's inequality give
\[
    \omega^{-1}(W(x))
    \le
    \int_0^T\eta(t)\omega^{-1}(\rho(t,x))\dd t
    \le Q(x).
\]
Thus,
\[
    -\Delta W\ge\omega^{-1}(W).
\]
If $W\not =0$, then, for every
$\gamma\in(0,W^*)$ where $W^*:=\sup_XW$, on
$X_\gamma:=\{x \in X \mid W(x)>W^*-\gamma\}$ one has
\[
    \Delta W
    \le-\omega^{-1}(W)
    <-\omega^{-1}(W^*-\gamma)<0.
\]
This contradicts \SCinf{} as characterized by the Omori--Yau maximum principle
in Proposition~\ref{prop:SCinf-characterizations}~(v).
Hence, $W=0$ and
\eqref{eq:averaged-modulus-DW} then gives $Q=0$.
Since $\eta>0$ and $q$ and $\rho$ are positive and continuous in time,
$Q=W=0$ implies $q=\rho=0$.
\end{proof}

We next prove the main result of this section. It gives uniqueness of bounded
solutions to the filtration equation under the stochastic completeness assumption.

\begin{theorem}[\SCinf{} implies uniqueness]\label{thm:killing-unique}
Assume that $G$ satisfies \SCinf{}.  Then, for every $T>0$, every $u_0\in\ellinfty$
and every $\phi \in \mathcal{I}$, the \ref{eq:GPME} with initial datum $u_0$ has at most one bounded pointwise solution on $[0,T]\times X$.
\end{theorem}

\begin{proof}
Let $v_1,v_2$ be bounded pointwise solutions of the \ref{eq:GPME} on $[0,T]\times X$ with the same initial datum $u_0$.  Choose constants
\[
    A\le\min\Bigl\{0,\inf_{[0,T]\times X}v_1,\inf_{[0,T]\times X}v_2\Bigr\}
    \qquad \text{and} \qquad
    B\ge\max\Bigl\{0,\sup_{[0,T]\times X}v_1,\sup_{[0,T]\times X}v_2\Bigr\}
\]
so that, in particular, $A\le\min\{0,\inf_Xu_0\}$ and $B\ge\max\{0,\sup_Xu_0\}$.
If $A=B$, then necessarily $A=B=0$, and the preceding bounds give
$v_1=v_2=0$. Thus, there is nothing to prove. Henceforth, assume $A<B$.
Let $u^A\le u^B$ be the extremal solutions of Theorem~\ref{thm:extremal} for this choice.  By the trapping property \eqref{eq:trapping} in Theorem~\ref{thm:extremal},
$u^A\le v_i\le u^B$ for $i=1,2$, so it suffices to show that $u^A=u^B$.

Set $I:=[A,B]$ and put
\[
    q:=u^B-u^A\ge0
    \qquad \text{and} \qquad
    \rho:=\phi(u^B)-\phi(u^A)\ge0 .
\]
Subtracting the equations satisfied by $u^A$ and $u^B$ and using the linearity of $\Delta$ on $\mathcal F$ gives
\begin{equation}\label{eq:q-rho-equation}
    \partial_t q+\Delta\rho=0
    \qquad \text{and } \qquad q(0,\cdot)=0 .
\end{equation}
Let $\omega$ be the concave modulus of $\phi|_I$ from Lemma~\ref{lem:modulus}.  Then,
\begin{equation}\label{eq:rho-modulus}
    0\le\rho\le\omega(q)
\end{equation}
and, in particular, $0\le\rho\le\omega(l)$, where $l$ is the length of $I$,
since $0\le q\le l$ and $\omega$ is increasing.
The functions $q$ and $\rho$ satisfy all the hypotheses of
Lemma~\ref{lem:averaged-modulus} with $R=l$; note that
$q(\cdot,x)$ is $C^1$ on $(0,T)$ for every $x$, hence locally absolutely continuous.  Therefore,
$q=\rho=0$, so $u^A=u^B$ and, therefore, $v_1=v_2$.
\end{proof}

\section{Nonuniqueness from incompleteness at infinity}\label{sec:nonunique-infinity}
In this section we prove the other half of our first main characterization by showing
that when stochastic completeness at infinity fails,
there exist infinitely many bounded solutions of the filtration equation.
In particular, when stochastic completeness at infinity fails, one can impose different constant boundary values at infinity and pass to the exhaustion limit.  A barrier argument, localized on the region where the defect function is close to its supremum, separates the resulting limits.
This suffices to establish non-uniqueness of solutions.

Throughout this section, we write
\[
    K:=\frac{\kappa}{\mu}
\]
so that $\Delta c=cK$ for any constant $c \in \R$.

We first need a variant on our various criteria for stochastic completeness at infinity.
\begin{lemma}\label{lem:defect-infinity}
Assume that $G$ does not satisfy \SCinf{}.  Then there exists $U\in\ellinfty$ such that
\[
    0\le U\le1,
    \qquad
    \sup_XU=1
    \qquad \text{and} \qquad
    (\Delta +1)U\le0.
\]
Consequently, the function $V:=2U-1$ satisfies
\[
    -1\le V\le1,
    \qquad
    \sup_XV=1
    \qquad \text{and} \qquad
    \Delta V\le -(V+1)-K.
\]
In particular, for every $\varepsilon\in(0,1)$, the super-level set $\{x \in X \mid V(x)>1-\varepsilon\}$ is nonempty and on it
\[
    -\Delta V\ge V+1+K>2-\varepsilon .
\]
\end{lemma}

\begin{proof}
By the subsolution
formulation recalled in Proposition~\ref{prop:SCinf-characterizations}~(iii) with $\beta=1$, failure of stochastic completeness at infinity gives a nontrivial $0\le U'\in\ellinfty$ satisfying
\[
    (\Delta +1)U'\le0.
\]
Normalize $U'$ to get $U$:
\[
    U:=\frac{U'}{\sup_XU'}.
\]
Then, $0\le U\le1$, $\sup_XU=1$ and $(\Delta +1)U\le0$ as claimed.

Since $\Delta 1=K$, for $V=2U-1$ we have
\[
    \Delta V=2\Delta U-\Delta 1
    \le -2U-K
    =-(V+1)-K.
\]
The remaining assertions follow from $\sup_XV=1$.
\end{proof}

Next, we prove the main result of this section. It gives the
existence of infinitely many bounded solutions of the filtration
equation in the case of the lack of stochastic completeness at infinity.

\begin{theorem}[Failure of \SCinf{} gives infinitely many solutions]\label{thm:killing-nonlinear-nonunique}
Assume that $G$ does not satisfy \SCinf{}.  Then, for
every $u_0\in\ellinfty$ and every $\phi \in \mathcal{I}$,
there exists an infinite family
$\{u^\alpha\}_{\alpha}$ of global bounded pointwise solutions
of the \ref{eq:GPME} with initial datum $u_0$ such that, for every $S>0$,
the restrictions $u^\alpha|_{[0,S]\times X}$ are pairwise distinct.
In particular, for every $T>0$, the Cauchy problem of the \ref{eq:GPME} has infinitely many
bounded pointwise solutions on $[0,T]\times X$.
\end{theorem}

\begin{proof}
Fix $u_0\in\ellinfty$ and $\phi\in\mathcal I$.
Let $V$ be as in Lemma~\ref{lem:defect-infinity}, so that
\begin{equation}\label{eq:killing-V-defect}
    \Delta V\le -(V+1)-K.
\end{equation}

Fix $\alpha\in\R$.  Define
\[
    A_\alpha:=\min\{0,\alpha,\inf_Xu_0\},
    \qquad
    B_\alpha:=\max\{0,\alpha,\sup_Xu_0\}
    \qquad \text{and} \qquad
    R_\alpha:=B_\alpha-A_\alpha.
\]
Then,
\[
    \{0, \alpha, u_0(x) \} \subseteq [A_\alpha, B_\alpha] \qquad \text{for all } x \in X.
\]
For this fixed nonlinearity, set
\[
    \beta_\alpha:=\phi(\alpha),
    \qquad
    \underline\beta_\alpha:=\phi(A_\alpha)
    \qquad \text{and} \qquad
    \overline\beta_\alpha:=\phi(B_\alpha).
\]

We now adapt the exhaustion construction of pointwise solutions in~\cite[proof of Theorem~3.4]{bianchi2026bounded} to an arbitrary constant exterior datum $\alpha$. We note that, unlike the construction in \cite{bianchi2026bounded} where $\alpha$ is an extremal exterior value,
we cannot use monotonicity but rather have to use a diagonal argument and compactness
to extract a solution.

\textbf{Step 1: Approximating solutions.}
Let $\Omega_n\uparrow X$ be an exhaustion of $X$ by finite sets.  On $\Omega_n$ consider the finite exterior-Dirichlet problem
\begin{equation}\label{eq:killing-finite-alpha}
    \left(\partial_t+\Delta_{\Omega_n}^{\alpha}\Phi\right) u^\alpha_n=0,
    \qquad
    u^\alpha_n(0,\cdot)=u_0
    \quad\text{on }\Omega_n ,
\end{equation}
which is a finite system of ODEs with continuous right-hand side and thus admits a local solution by Peano's theorem.  Since $A_\alpha\le\alpha\le B_\alpha$, $A_\alpha\le0\le B_\alpha$ and $\phi$ is increasing, the constants $A_\alpha$ and $B_\alpha$ are, respectively, a subsolution and a supersolution of~\eqref{eq:killing-finite-alpha} since, for $x \in \Om_n$, we have
\[
    \Delta_{\Omega_n}^{\alpha}\Phi A_\alpha(x)
    =
    \frac1{\mu(x)}
    \sum_{y\notin\Omega_n}w(x,y)
    \bigl(\phi(A_\alpha)-\phi(\alpha)\bigr)
    +K(x)\phi(A_\alpha)
    \le0,
\]
whereas
\[
    \Delta_{\Omega_n}^{\alpha}\Phi B_\alpha(x)
    =
    \frac1{\mu(x)}
    \sum_{y\notin\Omega_n}w(x,y)
    \bigl(\phi(B_\alpha)-\phi(\alpha)\bigr)
    +K(x)\phi(B_\alpha)
    \ge0.
\]
By the comparison principle of Lemma~\ref{lem:finite-comparison}, every
solution of~\eqref{eq:killing-finite-alpha} therefore remains in the
box $[A_\alpha,B_\alpha]^{\Omega_n}$ for all $t \geq 0$ and, hence, extends to all of
$t\in[0,\infty)$.

Extend $u_n^\alpha$ by the constant $\alpha$ on $X\setminus\Omega_n$
to get $E_{\Om_n}^\alpha u_n^\alpha$.  Set
\[
    C_\alpha:=\max_{s\in[A_\alpha,B_\alpha]}|\phi(s)|.
\]
For each fixed $x\in X$ and all $n$ with $x\in\Omega_n$,
\[
    |\partial_tu_n^\alpha(t,x)|
    \le
    \frac{2C_\alpha}{\mu(x)}\sum_{y\in X}w(x,y)
    +\frac{C_\alpha\kappa(x)}{\mu(x)}
\]
which is independent of $t\ge0$ and $n$.  Thus, at every vertex, the sequence is uniformly
bounded and Lipschitz equicontinuous on each compact time interval.  By
Arzel\`a--Ascoli and a diagonal extraction over the countable family
$X\times\N$, a subsequence converges, uniformly on $[0,m]$ at every
vertex and for every $m\in\N$, to a bounded function
$u^\alpha \colon [0,\infty)\times X\to[A_\alpha,B_\alpha]$.

To see that
$u^\alpha$ solves the \ref{eq:GPME}, fix $x$ and $t\ge0$, choose
$m\in\N$ with $m\ge t$, and pass to the limit in the integrated form
\[
    u^\alpha_n(t,x)=u_0(x)-\int_0^t\Delta\Phi\bigl(E^{\alpha}_{\Omega_n}u^\alpha_n(s,\cdot)\bigr)(x)\dd s
\]
by dominated convergence, first in the weighted sums (with summable dominating function $2C_\alpha w(x,\cdot)/\mu(x)$) and then in time.  In the resulting identity
\[
    u^\alpha(t,x)=u_0(x)-\int_0^t\Delta\Phi\bigl(u^\alpha(s,\cdot)\bigr)(x)\dd s
\]
the integrand is continuous in $s$  by dominated convergence again, so
$u^\alpha$ is a global bounded pointwise solution of the \ref{eq:GPME}
with initial datum $u_0$.

\textbf{Step 2: A localized two-sided barrier for time integrals.}
Fix $S>0$ and set
\[
    F^\alpha_{S,n}(x):=\int_0^S\phi(u^\alpha_n(t,x))\,\dd t
    \qquad \text{and} \qquad
    F^\alpha_S(x):=\int_0^S\phi(u^\alpha(t,x))\,\dd t.
\]
To keep track of the exterior datum $\alpha$, define the full-space function
\[
    \widehat F^\alpha_{S,n}
    :=\int_0^S
    \Phi\bigl(E_{\Omega_n}^{\alpha}u_n^\alpha(t,\cdot)\bigr)\,\dd t.
\]
On $\Omega_n$, this function agrees with $F^\alpha_{S,n}$.  On $X\setminus\Omega_n$, the function $E_{\Omega_n}^{\alpha}u_n^\alpha(t,\cdot)$ has the value $\alpha$ at every time, and hence
\[
    \widehat F^\alpha_{S,n}(x)
    =\int_0^S\phi(\alpha)\,\dd t
    =S\beta_\alpha.
\]
Consequently,
\[
    \widehat F^\alpha_{S,n}
    =E_{\Omega_n}^{S\beta_\alpha}F^\alpha_{S,n}
    \qquad\text{on }X.
\]
Integrating~\eqref{eq:killing-finite-alpha} over $[0,S]$ and interchanging the time integral with $\Delta$ now gives, for every $x\in\Omega_n$,
\begin{align*}
    0
    &=u_n^\alpha(S,x)-u_0(x)
      +\bigl(\pi_{\Omega_n}\Delta\widehat F^\alpha_{S,n}\bigr)(x)\\
    &=u_n^\alpha(S,x)-u_0(x)
      +\Delta_{\Omega_n}^{S\beta_\alpha}F^\alpha_{S,n}(x).
\end{align*}
The interchange is justified by the bound $|\Phi(E_{\Omega_n}^{\alpha}u_n^\alpha)|\le C_\alpha$ and the summability of $w(x,\cdot)$. Thus,
\begin{equation}\label{eq:killing-integrated-alpha}
    \Delta_{\Omega_n}^{S\beta_\alpha}F^\alpha_{S,n}(x)
    =
    u_0(x)-u^\alpha_n(S,x).
\end{equation}
Since both $u_0(x)$ and $u_n^\alpha(S,x)$ belong to $[A_\alpha,B_\alpha]$, their difference has absolute value at most $B_\alpha-A_\alpha=R_\alpha$. Therefore,
\begin{equation}\label{eq:killing-source-bound}
    \left|
        \Delta_{\Omega_n}^{S\beta_\alpha}F^\alpha_{S,n}(x)
    \right|
    \le
    R_\alpha
    \qquad \text{for }x\in\Omega_n.
\end{equation}
Finally, the bounds from Step~1 and the monotonicity of $\phi$ imply
\[
    \underline\beta_\alpha
    \le \phi(u_n^\alpha(t,x))
    \le \overline\beta_\alpha
    \qquad\text{for }(t,x)\in[0,S]\times\Omega_n.
\]
Integrating these inequalities over $[0,S]$ yields
\begin{equation}\label{eq:killing-F-bound}
    S\underline\beta_\alpha
    \le
    F^\alpha_{S,n}(x)
    \le
    S\overline\beta_\alpha
    \qquad \text{for } x\in\Omega_n.
\end{equation}

Fix $\varepsilon_*\in(0,1)$ and set
\[
    D_{n,\varepsilon_*}:=
    \Omega_n\cap\{x\in X \mid V(x)>1-\varepsilon_*\}.
\]
We regard $F^\alpha_{S,n}$ as a function on $D_{n,\varepsilon_*}$ with exterior datum
\[
    \eta_{n,\alpha,S}(y)
    :=
    \begin{cases}
        F^\alpha_{S,n}(y) & \text{if } y\in\Omega_n\setminus D_{n,\varepsilon_*}\\
        S\beta_\alpha & \text{if } y\in X\setminus\Omega_n
    \end{cases}
\]
so that
\[
    \Delta_{D_{n,\varepsilon_*}}^{\eta_{n,\alpha,S}}F^\alpha_{S,n}
    =
    \Delta_{\Omega_n}^{S\beta_\alpha}F^\alpha_{S,n}
    \qquad\text{on }D_{n,\varepsilon_*}.
\]

Choose constants $c_{\alpha,S},d_{\alpha,S}>0$, depending on $\alpha$, $S$, and $\varepsilon_*$ but not on $n$, so large that
\begin{equation}\label{eq:killing-cd-source-choice}
    Sc_{\alpha,S}(2-\varepsilon_*)>R_\alpha,
    \qquad
    Sd_{\alpha,S}(2-\varepsilon_*)>R_\alpha,
\end{equation}
\begin{equation}\label{eq:killing-cd-exterior-choice}
    \beta_\alpha-c_{\alpha,S}\varepsilon_*
    \le
    \underline\beta_\alpha,
    \qquad
    \beta_\alpha+d_{\alpha,S}\varepsilon_*
    \ge
    \overline\beta_\alpha,
\end{equation}
and
\begin{equation}\label{eq:killing-cd-kappa-choice}
    c_{\alpha,S}\ge\max\left\{0,\frac{\beta_\alpha}{2}\right\},
    \qquad
    d_{\alpha,S}\ge\max\left\{0,-\frac{\beta_\alpha}{2}\right\}.
\end{equation}
Define the barriers
\[
    Z_\alpha:=\beta_\alpha+c_{\alpha,S}(V-1)
    \qquad \text{and} \qquad
    Y_\alpha:=\beta_\alpha+d_{\alpha,S}(1-V).
\]

We first compare the exterior data.  If $y\in X\setminus\Omega_n$, since $V(y)\le1$, we get
\[
    SZ_\alpha(y)\le S\beta_\alpha=\eta_{n,\alpha,S}(y)\le SY_\alpha(y).
\]
If $y\in\Omega_n\setminus D_{n,\varepsilon_*}$, then $V(y)\le1-\varepsilon_*$ and~\eqref{eq:killing-cd-exterior-choice} together with~\eqref{eq:killing-F-bound} gives
\[
    SZ_\alpha(y)
    \le S\bigl(\beta_\alpha-c_{\alpha,S}\varepsilon_*\bigr)
    \le S\underline\beta_\alpha
    \le \eta_{n,\alpha,S}(y) = F_{S,n}^\alpha(y)
    \le S\overline\beta_\alpha
    \le S\bigl(\beta_\alpha+d_{\alpha,S}\varepsilon_*\bigr)
    \le SY_\alpha(y).
\]
Thus,
\begin{equation}\label{eq:killing-exterior-comparison}
    SZ_\alpha
    \le
    \eta_{n,\alpha,S}
    \le
    SY_\alpha
    \qquad\text{on }X\setminus D_{n,\varepsilon_*}.
\end{equation}

Next we compare Laplacians on $D_{n,\varepsilon_*}$. Recall that constants are not annihilated by $\Delta$ when $\kappa\not=0$, but instead $\Delta c=cK$
for $c\in\R$.  Writing
\[
    Z_\alpha=(\beta_\alpha-c_{\alpha,S})+c_{\alpha,S}V
\]
inequality~\eqref{eq:killing-V-defect} gives, pointwise on $X$,
\begin{align*}
    \Delta Z_\alpha
    &=c_{\alpha,S}\Delta V+(\beta_\alpha-c_{\alpha,S})K                                      \\
    &\le -c_{\alpha,S}(V+1)-c_{\alpha,S}K+(\beta_\alpha-c_{\alpha,S})K                         \\
    &=-c_{\alpha,S}(V+1)+(\beta_\alpha-2c_{\alpha,S})K                                       \\
    &\le -c_{\alpha,S}(V+1)
\end{align*}
where the last inequality follows from $K\ge0$ and~\eqref{eq:killing-cd-kappa-choice}.  On $D_{n,\varepsilon_*}$, we have $V+1>2-\varepsilon_*$, whence, by~\eqref{eq:killing-cd-source-choice} and~\eqref{eq:killing-source-bound},
\[
    \Delta(SZ_\alpha)
    <-Sc_{\alpha,S}(2-\varepsilon_*)
    <-R_\alpha
    \le
    \Delta_{D_{n,\varepsilon_*}}^{\eta_{n,\alpha,S}}F^\alpha_{S,n}
    \qquad\text{on }D_{n,\varepsilon_*}.
\]
Since, on $D_{n,\varepsilon_*}$, $\Delta(SZ_\alpha)$ is the finite-domain Laplacian of $SZ_\alpha|_{D_{n,\varepsilon_*}}$ with exterior datum $SZ_\alpha|_{X\setminus D_{n,\varepsilon_*}}$, the exterior comparison~\eqref{eq:killing-exterior-comparison} and Lemma~\ref{lem:finite-elliptic-comparison} yield
\[
    SZ_\alpha\le F^\alpha_{S,n}
    \qquad\text{on }D_{n,\varepsilon_*}.
\]

Similarly, writing
\[
    Y_\alpha=(\beta_\alpha+d_{\alpha,S})-d_{\alpha,S}V
\]
we obtain
\begin{align*}
    \Delta Y_\alpha
    &=-d_{\alpha,S}\Delta V+(\beta_\alpha+d_{\alpha,S})K                                      \\
    &\ge d_{\alpha,S}(V+1)+d_{\alpha,S}K+(\beta_\alpha+d_{\alpha,S})K                         \\
    &=d_{\alpha,S}(V+1)+(\beta_\alpha+2d_{\alpha,S})K                                       \\
    &\ge d_{\alpha,S}(V+1)
\end{align*}
again by~\eqref{eq:killing-cd-kappa-choice}.  Therefore, on $D_{n,\varepsilon_*}$,
by \eqref{eq:killing-cd-source-choice} and \eqref{eq:killing-source-bound}, we get
\[
    \Delta(SY_\alpha)
    >Sd_{\alpha,S}(2-\varepsilon_*)
    >R_\alpha
    \ge
    \Delta_{D_{n,\varepsilon_*}}^{\eta_{n,\alpha,S}}F^\alpha_{S,n}
\]
and the exterior comparison~\eqref{eq:killing-exterior-comparison} with Lemma~\ref{lem:finite-elliptic-comparison} yields
\[
    F^\alpha_{S,n}\le SY_\alpha
    \qquad\text{on }D_{n,\varepsilon_*}.
\]

Now, fix $x$ with $V(x)>1-\varepsilon_*$.  For all large $n$ along the subsequence of Step~1 we have $x\in D_{n,\varepsilon_*}$ while $F^\alpha_{S,n}(x)\to F^\alpha_S(x)$ by dominated convergence.  Since the barriers do not depend on $n$, we conclude that
\begin{equation}\label{eq:killing-F-trapped-localized}
    S\bigl(\beta_\alpha-c_{\alpha,S}(1-V(x))\bigr)
    \le
    F^\alpha_S(x)
    \le
    S\bigl(\beta_\alpha+d_{\alpha,S}(1-V(x))\bigr)
    \qquad \text{for }
    x\in\{V>1-\varepsilon_*\}.
\end{equation}

\textbf{Step 3: Separation of limits.}
The preceding construction and estimate apply to every $\alpha \in \R$.
Since $\phi\in\mathcal I$, its image $\phi(\R)$ is a nontrivial interval and hence contains infinitely many elements. Choose $\mathcal A\subseteq\R$ containing exactly one representative of each level set $\phi^{-1}(\{\tau\})$, $\tau\in\phi(\R)$. Then
\[
\phi_{\vert\mathcal A}\colon\mathcal A\longrightarrow\phi(\R)
\]
is bijective. In particular, $\mathcal A$ is infinite and
$\beta_{\alpha_1}\ne\beta_{\alpha_2}$ whenever
$\alpha_1\ne\alpha_2$ in $\mathcal A$.

Fix $\alpha_1,\alpha_2\in \mathcal{A},$ $\alpha_1\ne\alpha_2,$ and use the same $S>0$ and
$\varepsilon_*\in(0,1)$ for both parameters $c_{\alpha_k,S}$ and $d_{\alpha_k,S}$.  Since  $\beta_{\alpha_1}\ne\beta_{\alpha_2}$,
we may choose $\varepsilon\in(0,\varepsilon_*)$ so
small that the two intervals
\[
    \bigl[
        \beta_{\alpha_k}-c_{\alpha_k,S}\varepsilon,\;
        \beta_{\alpha_k}+d_{\alpha_k,S}\varepsilon
    \bigr]
    \qquad \text{for }k=1,2,
\]
are disjoint.  Since $\sup_XV=1$, there exists $x\in X$ with $V(x)>1-\varepsilon$ and~\eqref{eq:killing-F-trapped-localized} gives
\[
    \frac1S F^{\alpha_k}_S(x)
    \in
    \bigl[
        \beta_{\alpha_k}-c_{\alpha_k,S}\varepsilon,\;
        \beta_{\alpha_k}+d_{\alpha_k,S}\varepsilon
    \bigr]
    \qquad \text{for }k=1,2.
\]
Hence, $F^{\alpha_1}_S(x)\ne F^{\alpha_2}_S(x)$ and, therefore, the
restrictions of $u^{\alpha_1}$ and $u^{\alpha_2}$ to
$[0,S]\times X$ are distinct.  Since $S>0$ was arbitrary, the family
has the asserted separation property on every finite time interval.
\end{proof}

\section{Parabolic characterization}\label{sec:parabolic-characterization}

Combining the results of the two preceding sections gives the characterization announced in the introduction as Theorem~\ref{thm:mainA}.  Besides uniqueness for arbitrary bounded initial data, it shows that stochastic completeness at infinity can be detected using a single initial datum,
and even using only the zero initial datum. Furthermore, stochastic
completeness at infinity is also characterized by a global bounded comparison principle
which we establish here.

We start with a general lemma concerning the Laplacian applied to the positive
part of a function which will be used in the proof of the comparison principle. In particular,
for a function $f \colon X \to \R$, we let $f_+(x)=\max\{f(x),0\}$
denote the positive part of $f$.

\begin{lemma}\label{lem:positive-part-kato}
If $f\in\mathcal F$, then $f_+\in\mathcal F$ and
\begin{equation*}
    \Delta f_+\le \one_{\{f\ge0\}}\Delta f.
\end{equation*}
Moreover,
\begin{equation*}
    \Delta f_+\le0\qquad\text{on }\{f\le0\}.
\end{equation*}
\end{lemma}

\begin{proof}
Since $0\le f_+\le |f|$, the defining summability condition for $\mathcal F$ shows that $f_+\in\mathcal F$.
For every $a,b\in\R$,
\[
    a_+-b_+\le \one_{\{a\ge0\}}(a-b).
\]
Apply this with $a=f(x)$ and $b=f(y)$, multiply by $w(x,y)$, and sum over $y$
to get the inequality for the summation part of the Laplacian. For the killing term part, we use
$f_+(x)=\one_{\{f(x)>0\}}f(x)$. Combining these gives the first inequality.

If $f(x)\le0$, then
\[
\Delta f_+(x)=-\frac1{\mu(x)}\sum_{y\in X}w(x,y)f_+(y)\le0,
\]
which proves the second inequality.
\end{proof}

We now use our previous results and the lemma above to give a general characterization
of stochastic completeness at infinity in terms of bounded solutions
of the filtration equation.
\begin{theorem}[Characterization of \SCinf{}]\label{thm:killing}
Let \(G\) be a weighted graph.
The following statements are equivalent:
\begin{enumerate}[label=(\roman*)]
    \item[\textup{(i)}]
    \(G\) satisfies \SCinf{}.

    \item[\textup{(ii)}]
    For every \(T>0\), every \(u_0\in\ellinfty\) and every $\phi \in \mathcal{I}$, the~\ref{eq:GPME} with initial datum $u_0$ has a unique bounded pointwise solution on
    \([0,T]\times X\).

    \item[\textup{(iii)}]
    There exist \(T_0>0\), \(u_0\in\ellinfty\) and $\phi \in \mathcal{I}$
    such that the~\ref{eq:GPME} with initial datum $u_0$
    has at most one bounded pointwise solution
    on \([0,T_0]\times X\).

    \item[\textup{(iv)}]
    There exist \(T_0>0\) and $\phi \in \mathcal{I}$
    such that $u=0$ is the only bounded
    pointwise solution to the \ref{eq:GPME} on \([0,T_0]\times X\) with initial
    datum \(u_0=0\).

    \item[\textup{(v)}]
    There exist \(T_0>0\) and $\phi \in \mathcal{I}$
    such that the following global bounded comparison principle
    holds: If \(u,v \colon [0,T_0]\times X\to\R\) are bounded functions such
    that
    \[
        u(\cdot,x),v(\cdot,x)
        \in C([0,T_0])\cap C^1((0,T_0))
        \qquad\text{for every }x\in X
    \]
    and
    \[
        (\partial_t+\Delta\Phi)u
        \leq
        (\partial_t+\Delta\Phi)v
        \qquad\text{on }(0,T_0)\times X
    \]
    with
    \[
        u(0,\cdot)\leq v(0,\cdot),
    \]
    then
    \[
        u\leq v
        \qquad\text{on }[0,T_0]\times X.
    \]

    \item[\textup{(vi)}]
    For every $u_0\in\ellinfty$ and every $\phi\in\mathcal I$, the~\ref{eq:GPME} has a unique global bounded pointwise solution.
\end{enumerate}

If one of the preceding conditions holds, then the comparison principle
in \textup{(v)} holds for every \(T>0\).

If $G$ is not stochastically complete at infinity, then, for every $u_0\in\ellinfty$ and every $\phi\in\mathcal I$, there are infinitely many global bounded pointwise solutions whose restrictions to every $[0,T]\times X$, $T>0$, are pairwise distinct. For $u_0=0$, an infinite family can be chosen of one sign: nonnegative when $\phi$ is nonconstant on $[0,\infty)$, and nonpositive otherwise.
\end{theorem}

\begin{proof}
The implication (i) $\Longrightarrow$ (ii) follows from the
uniqueness given by Theorem~\ref{thm:killing-unique}, together with the
existence given by Theorem~\ref{thm:extremal}. The implication
(ii) $\Longrightarrow$ (iii) is immediate.

Suppose that (i) fails.
Theorem~\ref{thm:killing-nonlinear-nonunique} gives, for every initial
datum and every $\phi \in \mathcal{I}$, infinitely many global bounded pointwise solutions whose
restrictions to each positive finite time interval are pairwise
distinct.  Thus, no
$T_0> 0$, $u_0 \in \ell^\infty(X)$ and $\phi \in \mathcal{I}$
can satisfy the uniqueness
property in (iii), proving
(iii) $\Longrightarrow$ (i).

We next prove the equivalence of (i) and (iv). If (i) holds,
then the uniqueness given by Theorem~\ref{thm:killing-unique}, applied
to the zero initial datum, shows that \(u=0\) is the only bounded
pointwise solution with initial datum $u_0=0$.

Conversely, suppose that \(G\) is not stochastically complete at
infinity and let \(u_0=0\).
For $\phi \in \mathcal{I}$,
choose $\alpha\in\R$ with $\phi(\alpha)\neq 0.$
The sign of $\alpha$ determines the sign of the solution constructed below.
In the proof of Theorem~\ref{thm:killing-nonlinear-nonunique}, one has
$A_\alpha=\min\{0,\alpha\}$ and $B_\alpha=\max\{0,\alpha\}$, so the finite subgraph
construction with exterior datum \(\alpha\) remains in the interval
$[\min\{0,\alpha\},\max\{0,\alpha\}]$. It therefore produces a global bounded
one-sign pointwise solution \(u^\alpha\) satisfying
\[
    \min\{0,\alpha\}\le u^\alpha\le\max\{0,\alpha\}.
\]
The solution associated with \(\alpha=0\) is identically zero, whereas
the separation argument gives
\[
    u^\alpha\not=0
    \qquad\text{whenever }\phi(\alpha)\ne0.
\]
Consequently, (iv) fails for every \(T_0>0\). This proves the
equivalence of \({\rm(i)}\) and \({\rm(iv)}\). Moreover, since the
separation argument distinguishes \(u^{\alpha_1}\) from \(u^{\alpha_2}\)
whenever \(\alpha_1\neq\alpha_2\) and $\phi(\alpha_1)\neq \phi(\alpha_2)$,
at least one of the two half-lines contains infinitely many parameters with distinct $\phi$-values. The corresponding solutions form an infinite nonnegative or nonpositive family with zero initial datum. This proves the final statement of the theorem.

It remains to show the equivalence to the comparison characterization in (v).
Assume (i), fix \(T>0\) and let \(u,v\) satisfy the hypotheses in (v). Set
\[
    d:=u-v,
    \qquad
    z:=\phi(u)-\phi(v),
    \qquad
    q:=d_+
    \qquad \text{and} \qquad
    \rho:=z_+.
\]
Since \(\phi\) is increasing,
\begin{equation}\label{eq:positive_inclusion}
    \{d>0\}\subseteq\{z\geq 0\}.
\end{equation}

Applying Lemma~\ref{lem:positive-part-kato} to $z$ gives
\begin{equation}\label{eq:Kato_consequences}
\Delta \rho\le \one_{\{z\ge0\}}\Delta z  \qquad \text{ and } \qquad \Delta \rho \le0  \text{ on } \{z\le0\}.
\end{equation}

For every fixed \(x\), the scalar positive-part chain rule gives, for
almost every \(t\in(0,T)\),
\[
    \partial_tq(t,x)
    =
    \one_{\{d(t,x)>0\}}\partial_td(t,x).
\]

We claim that for every $x\in X,$
for almost every \(t\in(0,T)\),
$$\partial_tq+\Delta\rho
    \leq0.$$

Using the inclusion of the positive sets given in \eqref{eq:positive_inclusion} as well as
\eqref{eq:Kato_consequences}, and the
assumed differential inequality, we obtain for almost every $t\in \{d(\cdot,x)>0\}
\subseteq \{z \geq 0\},$

\[
\begin{aligned}
    \partial_tq+\Delta\rho
    &\leq
    \one_{\{d>0\}}
    \partial_td+\one_{\{z\geq 0\}}\Delta z=\partial_td+\Delta z
\\
    &=
        (\partial_t+\Delta\Phi) u
        - (\partial_t+\Delta\Phi)v
    \leq0.
\end{aligned}
\]
For almost every $t\in \{d(\cdot,x)\leq 0\},$ by $\{d\leq 0\}\subseteq\{z\leq 0\}$
and \eqref{eq:Kato_consequences},
$$\partial_tq+\Delta\rho\leq 0.$$ This proves the claim.

Moreover,
\(q(0,\cdot)=0\), since \(u(0,\cdot)\leq v(0,\cdot)\) by assumption.

Choose a compact interval \(I\subset\R\) of positive length $l$ containing
the ranges of both \(u\) and \(v\) and let \(\omega \geq 0\) be the strictly increasing concave
modulus supplied by Lemma~\ref{lem:modulus} for \(\phi|_I\). Then,
\begin{equation}\label{eq:comp-modulus}
    0\leq\rho\leq\omega(q)
\end{equation}
since on \(\{d>0\}\) this is the modulus bound~\eqref{eq:modulus-bound}, while
on \(\{d\leq0\}\) one has \(\rho=0\).  In particular,
\(0\leq\rho\leq\omega(l)\), since \(0\leq q\leq l\).
For each $x$, the function $q(\cdot,x)$ is locally absolutely
continuous on $(0,T)$ by the scalar positive-part chain rule, while
$\rho(\cdot,x)$ is continuous.  Hence,
Lemma~\ref{lem:averaged-modulus}, applied with $R=l$, gives
$q=\rho=0$.  In particular, \(u\leq v\). Hence,
(i) $\Longrightarrow$ (v), and the same proof shows that the
comparison principle holds for every \(T>0\).

Conversely, assume (v). If \(u_1,u_2\) are two bounded
pointwise solutions with the same initial datum, then the differential
inequality in (v) holds with equality. Applying the comparison
twice gives
\[
    u_1\leq u_2
    \qquad\text{and}\qquad
    u_2\leq u_1.
\]
Thus, \(u_1=u_2\), so (iii) holds. We have already proved that
(iii) $\Longrightarrow$ (i), completing the equivalence of conditions~\textup{(i)}--\textup{(v)}.

The finite- and global-time multiplicity assertions follow from
Theorem~\ref{thm:killing-nonlinear-nonunique}; the one-sign refinement for
\(u_0=0\) was established in the proof of the equivalence of \textup{(i)} and \textup{(iv)} above.

Finally, if (i) holds, Theorem~\ref{thm:extremal} supplies a
global bounded pointwise solution.  Any two such solutions restrict,
for every \(m\in\N\), to bounded pointwise solutions on
\([0,m]\times X\) and, therefore, coincide there by
Theorem~\ref{thm:killing-unique}.  Hence, the global solution is unique.
Conversely, condition~\textup{(vi)} implies~\textup{(i)}, since failure of~\textup{(i)} gives infinitely many global bounded solutions for every $u_0\in\ellinfty$ and every $\phi\in\mathcal I$ by Theorem~\ref{thm:killing-nonlinear-nonunique}.
This completes the proof of all of the statements.
\end{proof}

\begin{remark}
If $\kappa = 0$, stochastic completeness at infinity is ordinary stochastic completeness by definition.  Thus, Theorem~\ref{thm:killing} contains the parabolic characterization of stochastic completeness, the graph analogue of~\cite{GIM,GIMP}, as the special case $\kappa = 0$.
\end{remark}

\section{Generalized conservation of mass}
\label{sec:mass-characterization}

In this section, we prove that stochastic completeness at infinity is equivalent to a generalized conservation of mass for the filtration equation: the total mass of a positive pointwise solution at time $t$, augmented by the mass dissipated by the killing term up to time $t$, equals the initial mass.
The dissipated mass is measured by the nonlinear flux $\Phi u$ rather than by $u$ itself, and no summability assumption on $\kappa$ is required: stochastic completeness at infinity alone provides the integrability, through a no-flux identity at infinity.

Throughout this section the killing term $\kappa\ge0$ is arbitrary.
As in Section~\ref{sec:nonunique-infinity}, we write
\[
    K:=\frac{\kappa}{\mu}.
\]
We denote by $\langle f,g\rangle:=\sum_{x\in X}f(x)g(x)\mu(x)$
the inner product on $\ell^2(X,\mu).$
For positive functions $f, g\colon X \to [0,\infty)$, we will also
write $\langle f,g\rangle=\sum_{x\in X}f(x)g(x)\mu(x)$ with the understanding
that this could be infinite.
A superscript $+$, as in $\ell^{1,+}(X,\mu)$ and $\ell^{\infty,+}(X)$, indicates the positive elements of these sets.  For a positive function $v$ on $X$, set
\[
    \mathsf M_v:=\sum_{x\in X}v(x)\mu(x)\in[0,\infty]
\]
which we refer to as the \emph{mass} of $v$.
For a positive solution $u$ to the \ref{eq:GPME},
we write $\mathsf M_u(t):=\mathsf M_{u(t,\cdot)}$.

When $\kappa=0$ and $\Phi = \id$, stochastic completeness expresses conservation of heat, and one expects it to be detected by conservation of the mass $\mathsf M_u$ when $u$ is a positive
solution of the heat equation.
When $\kappa\not=0$, the killing term removes mass from the system, so that $\mathsf M_u$ decreases even on the most well-behaved graphs. The relevant question is then no longer whether the mass is conserved, but whether \emph{all} of the dissipated mass is accounted for by the killing term or whether some additional mass is lost at, or created from, infinity.  This question is governed by stochastic completeness at infinity, which we quantify at the nonlinear level by an exact mass balance: on graphs of finite measure within the full class of bounded positive solutions and on graphs of arbitrary measure for finite-mass data.

We first recall the linear picture.  Let $L$ denote the minimal positive self-adjoint realization
of $\Delta$ in $\ell^2(X,\mu)$ associated with the regular Dirichlet form of the graph, see~\cite[Chapter~1]{KLW21}, and let
\[
    P_t:=e^{-tL}
    \qquad \text{and} \qquad
    R_\beta:=(\beta+L)^{-1}
    \qquad \text{for } t,\beta>0
\]
be the minimal heat semigroup and its resolvent; we use the same symbols for their consistent sub-Markovian extensions to $\ell^1(X,\mu)$ and $\ellinfty$.  Following Keller--Lenz \cite{KL12}, stochastic completeness at infinity \SCinf{} is a generalized conservation property given by
\begin{equation}\label{eq:MS-generalized-conservation}
    P_t 1+\int_0^tP_sK\dd s=1
    \qquad \text{for } t>0
\end{equation}
where the second summand is understood in the extended positive sense.
The resolvent counterpart to \eqref{eq:MS-generalized-conservation} is
\begin{equation}\label{eq:generalized-resolvent-conservation}
    1=\beta R_\beta1+R_\beta K
    \qquad \text{for } \beta>0
\end{equation}
again in the extended positive sense; the precise exhaustion--truncation construction of
$R_\beta 1$ and $R_\beta K$, which is needed when $\mu(X)=\infty$, is recalled at the beginning of the proof of Lemma~\ref{lem:no-flux-killing}.  By general theory, \eqref{eq:MS-generalized-conservation}, \eqref{eq:generalized-resolvent-conservation}, and the conditions in Proposition~\ref{prop:SCinf-characterizations} are equivalent.

We shall repeatedly use the standard properties of the resolvent $R_\beta$
collected in \cite[Section~2.1 and the discussion preceding Theorem~7.2]{KLW21}.
On $\ell^2(X,\mu)$ the operator $R_\beta$ is positivity preserving and
self-adjoint, its extensions to the spaces $\ell^p(X,\mu)$ are consistent, and
$\beta R_\beta$ is a contraction on each $\ell^p(X,\mu)$, $1\le p\le\infty$.
If $q\in\ell^2(X,\mu)$, then $R_\beta q\in\dom(L)$ and $(\beta+L)R_\beta q=q$.
The action of $R_\beta$ on arbitrary nonnegative functions is defined via the
monotone extension of \cite[Section~7.2]{KLW21}; in particular, if
$q_n\uparrow q$ pointwise, then $R_\beta q_n\uparrow R_\beta q$ in the extended
positive sense.

For the linear heat equation, i.e., when $\phi(s)=s$,
pairing~\eqref{eq:MS-generalized-conservation} with a positive $f\in\ell^1(X,\mu)\cap\ellinfty$ and using the self-adjointness of $P_t$ together with monotone convergence gives
that stochastic completeness at infinity is equivalent to the mass balance
\[
    \sum_{x\in X}P_tf(x) \mu(x)
    +\int_0^t\sum_{x\in X}\kappa(x)P_sf(x)\dd s
    =\sum_{x\in X}f(x) \mu(x)
\]
which gives that the mass remaining at time $t$, plus the mass removed by the killing term up to time $t$, equals the initial mass. We note that when $u_0$ is positive, $u(t,x)=P_t u_0(x)$
gives the minimal positive pointwise solution of the heat equation in the linear case,
see Lemma~1.24 in \cite{KLW21}.

The nonlinear substitute requires one structural change.  In the filtration equation the killing term acts on $\Phi u$ rather than on $u$: pointwise,
\[
    \partial_tu(t,x)
    =
    -\Delta\Phi u(t,x)
    =
    -\frac1{\mu(x)}\sum_{y\in X}w(x,y)\bigl(\phi(u(t,x))-\phi(u(t,y))\bigr)
    -K(x)\,\phi(u(t,x))
\]
so the instantaneous rate at which the killing term removes mass at the vertex $x$ is $\kappa(x)\phi(u(t,x))$ and not $\kappa(x)u(t,x)$.

This leads to the notion
of generalized mass balance \MB{} for positive solutions $u$ as introduced in Definition~\ref{def:generalized-mass-conservation},
i.e.,
$$\sum_{x\in X}u(t,x)\mu(x) +\int_0^t\sum_{x\in X}\kappa(x)\,\phi(u(s,x))\dd s
=\sum_{x\in X}u_0(x)\mu(x)
    \qquad \text{for all } 0\le t\le T$$
or, in the notation of this section,
$$\mathsf M_u(t) +\int_0^t\sum_{x\in X}\kappa(x)\,\phi(u(s,x))\dd s =\mathsf M_{u_0}.$$

For the statements on graphs of infinite measure, we recall the linear growth condition \LG{}.
Since $\phi\in\mathcal{I},$ condition \LG{} is equivalent to
\begin{equation*}
    C_R:=\sup_{0<r\le R}\frac{\phi(r)}{r}<\infty
    \qquad\text{for every }R>0.
\end{equation*}
Only positive arguments occur below; for signed solutions, the natural replacement is the corresponding two-sided condition $\limsup_{r\to0}|\phi(r)|/|r|<\infty$.  For the signed powers $\phi(r)=r|r|^{m-1}$, condition \LG{} holds exactly when $m\ge1$.  Thus, the finite measure statements below cover all $m>0$, whereas the arbitrary measure statement covering every bounded finite-mass datum includes the linear and porous medium ranges $m\ge1$; the fast diffusion obstruction is exhibited in Example~\ref{ex:sharpness-mass-extensions}\,\textup{(a)}.  Finally, the natural class of bounded finite-mass data is
\[
    \mathcal D_+:=\ell^{1,+}(X,\mu)\cap\ellinfty.
\]
If $\mu(X)<\infty$, then $\mathcal D_+=\ell^{\infty,+}(X)$.

The following no-flux identity is the main analytic input.  It allows us to derive the global balance without any summability assumption on $\kappa$: stochastic completeness at infinity itself provides the required integrability.
This extends the Green's formula characterizations of stochastic completeness
at infinity found in~\cite{HKLMS17}, see also Corollary~7.27 in~\cite{KLW21}.

\begin{lemma}
\label{lem:no-flux-killing}
Assume that $G$ satisfies \SCinf{}.  If
\[
    f,\Delta f\in\ell^1(X,\mu)\cap\ellinfty,
\]
then $K \cdot  f\in\ell^1(X,\mu)$ and
\begin{equation}\label{eq:no-flux-killing}
    \sum_{x\in X}\Delta f(x)\mu(x)
    =
    \sum_{x\in X}\kappa(x)f(x)
\end{equation}
where both series converge absolutely.
\end{lemma}

\begin{proof}
Fix $\beta>0$.  Stochastic completeness at infinity is equivalent to the generalized resolvent-conservation identity
\begin{equation}\label{eq:extended-resolvent-conservation}
    1=\beta R_\beta 1+R_\beta K
\end{equation}
which, in particular, implies the inequalities
\begin{equation}\label{eq:extended-resolvent-inequalities}
    0\le R_\beta K\le 1
    \end{equation}
where all terms are understood in the extended positive sense.  More precisely, let $(\Om_n)$ be an exhaustion of $X$ by finite sets and define
\[
    R_\beta 1:=\lim_{n\to\infty}R_\beta\one_{\Om_n}.
\]
We note that this limit is independent of the exhaustion by a maximum principle
as in Proposition~1.20 in \cite{KLW21}.  If
\[
    K_n:=(K\wedge n)\,\one_{\Om_n},
\]
then $K_n\in\ell^1(X,\mu)\cap\ell^2(X,\mu)\cap\ellinfty$, the functions $R_\beta K_n$ increase pointwise to $R_\beta K$, and~\eqref{eq:extended-resolvent-conservation} holds pointwise.  This is the resolvent form of the generalized conservation characterization in~\cite[Theorem~7.2]{KLW21} and~\cite{KL12}; the spatial cutoff in the definition of $K_n$ is needed when $\mu(X)=\infty$.

Set
\[
    g:=(\beta +\Delta) f.
\]
Then, by the assumptions on $f$ and $\Delta f$, $g\in\ell^1(X,\mu)\cap\ellinfty\subseteq\ell^2(X,\mu)$.  The function $v:=R_\beta g$ belongs to $\ell^1(X,\mu)\cap\ellinfty$ because $\beta R_\beta$ is sub-Markovian and contracts $\ell^1(X,\mu)$. Furthermore, $v$ satisfies
\[
    (\beta+\Delta)v=g
\]
pointwise on $X$, since the generator $L$ is a restriction of the formal Laplacian $\Delta$, see~\cite[Theorem~1.6]{KLW21}.  Consequently, $h:=f-v$ is bounded and satisfies $(\beta+\Delta)h=0$.  The injectivity of $\beta+\Delta$ on $\ellinfty$, which is equivalent to \SCinf{}, see Proposition~\ref{prop:SCinf-characterizations}, gives $h=0$, that is,
\begin{equation}\label{eq:f-resolvent-g}
    f=R_\beta g.
\end{equation}

Positivity of the resolvent and~\eqref{eq:f-resolvent-g} yield
\[
    |f|\le R_\beta|g|.
\]
Using the truncations $K_n$, the self-adjointness of $R_\beta$ on $\ell^2(X,\mu)$ and monotone convergence, we obtain
\begin{align*}
    \displaystyle\sum_{x\in X}\kappa(x)|f(x)| \mu(x)
    &=\langle K,|f|\rangle
    \le\langle K,R_\beta|g|\rangle
    =\lim_{n\to\infty}\langle K_n,R_\beta|g|\rangle \\
    &=\lim_{n\to\infty}\langle R_\beta K_n,|g|\rangle
    \le\|g\|_1<\infty,
\end{align*}
since $0\le R_\beta K_n\le R_\beta K\le 1$ by~\eqref{eq:extended-resolvent-inequalities}.  Thus, $K \cdot f\in\ell^1(X,\mu)$.

It remains to identify the total flux.  We pair~\eqref{eq:extended-resolvent-conservation} with $g$: since $\beta R_\beta 1$ and $R_\beta K$ take values in $[0,1]$ and
$g\in\ell^1(X,\mu)$, each pairing converges absolutely and
\begin{align*}
    \langle g,1\rangle
    &=\beta\langle g,R_\beta 1\rangle+\langle g,R_\beta K\rangle \\
    &=\beta\langle R_\beta g,1\rangle+\langle R_\beta g,K\rangle \\
    &=\beta\langle f,1\rangle+\langle f,K\rangle .
\end{align*}
Here, the first self-adjointness identity is justified by replacing $1$ with $\one_{\Om_n}$ and passing to the limit by dominated convergence, with dominating functions $|g|$ and $|f|\in\ell^1(X,\mu)$; the second uses the truncations $K_n$ in the same way, now with dominating functions $|g|$ and $|f|K\in\ell^1(X,\mu)$, the integrability established above.  Since $g=(\beta +\Delta) f$, the terms containing $\beta$ cancel, and we are left with
\[
    \langle\Delta f,1\rangle=\langle f,K\rangle
\]
which is precisely~\eqref{eq:no-flux-killing}.  The hypotheses give absolute convergence of the left-hand side and the equivalence above gives it for the right-hand side.
\end{proof}

We next state a general consequence of the generalized mass balance for positive
solutions with initial condition in $\ell^1$. For $T > 0$, we
write $AC([0,T])$ for the absolutely continuous functions on $[0,T]$.

\begin{lemma}[Differential form of the generalized mass balance]\label{lem:differential-mass-balance}
Let $u$ be a positive pointwise solution of the~\ref{eq:GPME} on $[0,T]\times X$ with initial datum $u_0\in\ell^{1,+}(X,\mu)$. If $u$ satisfies~\MB{}, then
\[
q(t):=\sum_{x\in X}\kappa(x)\phi(u(t,x))\in L^1(0,T),
\]
$\mathsf M_u\in AC([0,T])$, and
\[
\mathsf M_u'(t)=-q(t)
\]
for almost every $t\in(0,T)$.
\end{lemma}

\begin{proof}
The function $q$ is measurable as the increasing limit of continuous finite partial sums. The balance at $t=T$ gives
\[
\int_0^Tq(s)\dd s=\mathsf M_{u_0}-\mathsf M_u(T)\le\mathsf M_{u_0}<\infty.
\]
Consequently,
\[
\mathsf M_u(t)=\mathsf M_{u_0}-\int_0^tq(s)\dd s,
\]
which gives the asserted absolute continuity and derivative identity.
\end{proof}

The next proposition derives some consequences of stochastic completeness
at infinity. In particular, when a solution is additionally continuously in $\ell^1$,
the solution will satisfy the generalized mass balance.

\begin{proposition}
\label{prop:solution-wise-mass-criterion}
Assume that $G$ satisfies \SCinf{}.  Let $u$ be a bounded positive pointwise solution of
the \ref{eq:GPME}
on $[0,T]\times X$ such that $u(t,\cdot)\in\ell^1(X,\mu)$ for every $t\in[0,T]$ and assume that
\begin{equation}\label{eq:solution-wise-integrated-flux}
    F_t:=\int_0^t\Phi u(s,\cdot)\dd s\in\ell^1(X,\mu)
    \qquad \text{for }0\le t\le T.
\end{equation}
Then $u$ satisfies the generalized mass balance~\MB{} and $\mathsf M_u$
satisfies the differential form of the generalized mass balance: $\mathsf M_u\in AC([0,T])$ and
\begin{equation}\label{eq:differential-mass-balance}
 \mathsf M_u'(t)
    =
    -\sum_{x\in X}\kappa(x)\,\phi(u(t,x))
\end{equation}
for almost every $t\in(0,T)$.
\end{proposition}

\begin{proof}
For every fixed $x\in X$, the boundedness of $\phi(u)$ and the
summability of $w(x,\cdot)$ permit us to interchange $\Delta$ with the
time integral.  Integrating the equation first away from the time
endpoints and then using the continuity of $u(\cdot,x)$ gives
\begin{equation}\label{eq:integrated-pointwise-equation}
    \Delta F_t
    =
    \int_0^t\Delta\Phi u(s,\cdot)\dd s
    =
    u_0-u(t,\cdot).
\end{equation}
Thus, $F_t,\Delta F_t\in\ell^1(X,\mu)\cap\ellinfty$.
Summing \eqref{eq:integrated-pointwise-equation},
Lemma~\ref{lem:no-flux-killing} and Tonelli's theorem yield
\[
    \mathsf M_{u_0}-\mathsf M_u(t)
    =\sum_{x\in X}\Delta F_t(x) \mu(x)
    =\sum_{x\in X}\kappa(x)F_t(x)
    =\int_0^t\sum_{x\in X}\kappa(x)\,\phi(u(s,x))\dd s
\]
which establishes \MB{}.
The conclusions for $\mathsf M_u$ now follow directly from Lemma~\ref{lem:differential-mass-balance}.
\end{proof}

We can now state a mass characterization result, which covers both finite- and
infinite-measure graphs. In the infinite measure case, we need an additional
assumption on $\Phi$ as formulated in \LG{}.
Furthermore, we now add the assumption
that $\phi^{-1}(\R_+)\neq\emptyset$ where $\R_+=(0,\infty)$.
This is needed as we consider positive solutions
and we require at least two positive solutions when the graph does not satisfy
stochastic completeness at infinity. The assumption $\phi^{-1}(\R_+)\neq\emptyset$ is
used for the separation of limits argument in Step~3 of the proof of
Theorem~\ref{thm:killing-nonlinear-nonunique}.
We note that, as $\phi$ is assumed to be continuous, as soon as $\phi^{-1}(\R_+)\neq\emptyset$,
$\phi$ will actually take on infinitely many positive values.

\begin{theorem}[Stochastic completeness and mass balance]
\label{thm:generalized-mass-characterization} Let $\phi\in\mathcal{I}$ with $\phi^{-1}(\R_+)\neq\emptyset.$
Assume that either
\begin{equation*}
    \mu(X)<\infty
    \qquad\text{or}\qquad
    \phi\text{ satisfies }\LG{}.
\end{equation*}
The following statements are equivalent:
\begin{enumerate}
    \item[\rm(i)]
    \(G\) satisfies \SCinf{}.

    \item[\rm(ii)]
    For every \(T>0\), every bounded positive pointwise solution \(u\) of the \ref{eq:GPME} with initial datum \(u_0 \in \mathcal D_+\) satisfies \MB{} on \([0,T]\).

    \item[\rm(iii)]
    There exist \(T_0>0\) and a nonzero positive finitely supported
    \(u_*\) such that every bounded positive pointwise solution on
    \([0,T_0]\times X\) with initial datum \(u_*\)
    satisfies~\MB{}.

    \item[\rm(iv)]
    For some \(T_0>0\), every bounded positive pointwise solution on
    \([0,T_0]\times X\) with zero initial datum
    satisfies~\MB{}.
\end{enumerate}

If these conditions hold, then, for every solution \(u\) as in
\textup{(ii)}, the total mass \(\mathsf M_u\) is absolutely continuous
on \([0,T]\) and satisfies~\eqref{eq:differential-mass-balance} for
almost every \(t\in(0,T)\).

If \(G\) is not stochastically complete at infinity, then, for every
\(T>0\) and every \(u_0\in\mathcal D_+\), at least one bounded
positive pointwise solution with initial datum \(u_0\)
violates~\MB{}.  For \(u_0=0\), every
nontrivial bounded positive solution violates \MB{}.

When \(\mu(X)<\infty\), statement \textup{(iii)} may equivalently be
replaced by a constant-data test: There exist \(T_0>0\) and \(c_0 \in \R_+\)
such that every bounded positive solution with \(u_0= c_0\)
satisfies \MB{}.  Moreover, in the incomplete case, the violating
solution may be chosen for every constant datum \(u_0= c\), for
\(c\ge0\).
\end{theorem}

\begin{proof}
We divide the proof into two steps.

\medskip
\noindent\textbf{Step 1: \textup{(i)} implies \textup{(ii)}
and the additional statements when (i) holds.}
Assume that (i) holds,
let \(u\) be a bounded positive pointwise solution of the \ref{eq:GPME} on
\([0,T]\times X\) with \(u_0\in\mathcal D_+\), and set
\[
    R:= \sup_{(t, x) \in [0,T] \times X} |u(t,x)|
    \qquad \text{and} \qquad
    \widehat R:=\max\{1,R\}.
\]
We first prove that
\begin{equation}\label{eq:pointwise-l1-bound}
    \|u(t,\cdot)\|_1\le\|u_0\|_1
    \qquad \text{for }0\le t\le T.
\end{equation}
As we assume \SCinf{}, we note that the \ref{eq:GPME} has a unique
bounded solution for bounded initial data by Theorem~\ref{thm:killing-unique}.
Let \((\Omega_n)\) be an exhaustion
of \(X\) and let \(u_n\) solve the finite \ref{eq:GPME} problem on \(\Omega_n\), with
initial value \(u_0|_{\Omega_n}\) and zero exterior value; denote its
extension by zero as \(U_n\).  The comparison principle of
Lemma~\ref{lem:finite-comparison} gives
\[
    0\le U_n\le U_{n+1}\le R_0:=\|u_0\|_\infty .
\]
The zero-Dirichlet exhaustion construction of
Theorem~\ref{thm:extremal}, with \(A=0\), therefore produces the minimal
positive bounded solution \(u^0=\lim_nU_n\) and the uniqueness of
bounded pointwise solutions under stochastic completeness at infinity,
Theorem~\ref{thm:killing-unique}, gives
\[
    u=u^0 .
\]

Summing the finite equation over \(\Omega_n\), the internal edge terms
cancel by the symmetry of \(w\) and \(\phi(0)=0\), whence
\begin{equation}\label{eq:finite-mass-dissipation}
    \partial_t
       \sum_{x\in\Omega_n}u_n(t,x)\mu(x)
    =-\sum_{\substack{x\in\Omega_n\\y\notin\Omega_n}}
       w(x,y)\,\phi(u_n(t,x))
      -\sum_{x\in\Omega_n}\kappa(x)\,\phi(u_n(t,x))
    \le0.
\end{equation}
Thus,
\[
    \sum_{x\in X}U_n(t,x)\mu(x)
    \le\sum_{x\in\Omega_n}u_0(x)\mu(x)
    \le\|u_0\|_1
\]
and monotone convergence proves~\eqref{eq:pointwise-l1-bound}.

For \(t\in[0,T]\), define
\[
    F_t(x):=\int_0^t\phi(u(s,x))\dd s.
\]
Then, \(0\le F_t\le t\,\phi(R)\).  If \(\mu(X)<\infty\), this boundedness
implies \(F_t\in\ell^1(X,\mu)\), without any growth assumption on
\(\phi\).

If \(\mu(X)=\infty\), the \LG{} condition, which reads
as $C_{\widehat{R}}:=\sup_{0<r\le \widehat{R}} {\phi(r)}/{r}<\infty$
and \(0\le u\le R\le\widehat R\) give
\(\phi(u)\le C_{\widehat R}u\).  Hence,
Tonelli's theorem and \eqref{eq:pointwise-l1-bound} yield
\begin{equation}\label{eq:Ft-L1-LG}
    \|F_t\|_1
    \le C_{\widehat R}\int_0^t\|u(s,\cdot)\|_1\dd s
    \le C_{\widehat R}\,t\,\|u_0\|_1 .
\end{equation}
Therefore, \(F_t\in\ell^1(X,\mu)\cap\ellinfty\) in either case.

All the hypotheses of
Proposition~\ref{prop:solution-wise-mass-criterion} are therefore
satisfied which establishes~\MB{}, together with the asserted
absolute continuity and differential conclusions.  Since the killing
integral in~\MB{} is non-negative, the balance also
proves~\eqref{eq:pointwise-l1-bound} in the finite-measure case.  This
completes the proof of \textup{(ii)}.

\medskip
\noindent\textbf{Step 2: The converse implications
and additional statements when (i) does not hold.}
The implications \textup{(ii)} \(\Longrightarrow\) \textup{(iii)} and
\textup{(ii)} \(\Longrightarrow\) \textup{(iv)} are immediate.  We now prove
(iii) $\Longrightarrow$ (i) and (iv) $\Longrightarrow$ (i) by contraposition.

Assume that \(G\) is not stochastically complete at infinity and fix
\(T>0\) and \(u_0\in\mathcal D_+\).  In the finite subgraph construction of
the bounded solution with initial datum $u_0$
used in the proof of Theorem~\ref{thm:killing-nonlinear-nonunique},
choose two constant exterior values
\[
    0=\alpha_1<\alpha_2 \text{ such that } \phi(\alpha_2)>0
\]
where we use the assumption $\phi^{-1}(\R_+)\neq\emptyset.$
By the comparison principle of Lemma~\ref{lem:finite-comparison}, the
finite subgraph solutions are ordered, and a common diagonal extraction
produces two bounded positive pointwise solutions with initial datum
\(u_0\) satisfying
\[
    0\le u_1\le u_2
    \qquad\text{on }[0,T]\times X.
\]
Since \(0=\phi(\alpha_1)\neq\phi(\alpha_2)\), the separation argument of
Step~3 of that proof gives \(u_1\not= u_2\). Since their initial data
agree, there exist \(t_*\in(0,T]\) and \(x_*\in X\) with
\[
    u_1(t_*,x_*)<u_2(t_*,x_*).
\]
Put
\[
    D_j(t):=\int_0^t\sum_{x\in X}
       \kappa(x)\,\phi\bigl(u_j(s,x)\bigr)\dd s
    \qquad \text{for } j=1,2.
\]
Then, \(D_1\le D_2\), by the monotonicity of \(\phi\).  If one of
\(\mathsf M_{u_j}(t_*)\) or \(D_j(t_*)\) is infinite, then $u_j$
already violates \MB{}, whose right-hand side
\(\mathsf M_{u_0}\) is finite.  Otherwise, strict positivity at \(x_*\)
and \(\mu(x_*)>0\) give
\[
    \mathsf M_{u_1}(t_*)+D_1(t_*)
    <
    \mathsf M_{u_2}(t_*)+D_2(t_*)
\]
so the two quantities cannot both equal the common initial mass
\(\mathsf M_{u_0}\).  Thus, at least one of the two solutions
violates~\MB{}.

This argument applies to any prescribed \(u_0\in\mathcal D_+\), in
particular, to the datum \(u_*\) in \textup{(iii)}.  Hence
\textup{(iii)} implies \textup{(i)}.

If the initial datum is zero and stochastic completeness at infinity fails,
Theorem~\ref{thm:killing}~(iv) supplies a
nontrivial bounded solution \(u\) which can be chosen positive
by the additional assumption $\phi^{-1}(\R_+)\neq\emptyset$.
Hence, since $u(t_0,x_0)>0$ for some $t_0>0$ and $x_0 \in X$, we get
\[
    \mathsf M_u(t_0)
    +\int_0^{t_0}\sum_{x\in X}\kappa(x)\,\phi(u(s,x))\dd s
    >0=\mathsf M_{u_0}
\]
so every nontrivial zero-data solution violates \MB{}.  This
proves \textup{(iv)} \(\Longrightarrow\) \textup{(i)} and all the assertions
in the incomplete case.

If \(\mu(X)<\infty\), every non-negative constant belongs to
\(\mathcal D_+\).  The same contraposition with \(u_0= c>0\) proves
the constant-data version of \textup{(iii)}, and the preceding argument
also covers every \(c\ge0\).
\end{proof}

When \(\kappa =0\), the identity~\MB{} reduces to the conservation of the total mass, \(\mathsf M_u(t)=\mathsf M_{u_0}\), and Theorem~\ref{thm:generalized-mass-characterization} becomes a mass characterization of ordinary stochastic completeness.  On graphs of finite measure the positive constants are then themselves solutions, and stochastic incompleteness produces the stronger, two-sided phenomenon of simultaneous loss and creation of mass from the same constant initial datum.  This refinement does not carry over verbatim to \(\kappa\not=0\): its proof rests on the fact that the positive constants are solutions, which fails in the presence of killing, since
\[
    \Delta\Phi c=K\cdot\phi(c)\not =0
    \qquad\text{for }c>0 .
\]
For arbitrary \(\kappa\ge0\), Theorem~\ref{thm:generalized-mass-characterization} still gives creation of mass from the zero initial datum, and failure of the generalized balance for at least one bounded positive solution issuing from every datum in \(\mathcal D_+\).

We now make explicit the statements about the no-killing and finite measure case.

\begin{corollary}
\label{cor:mass-no-killing}
Let $G$ be a graph with $\kappa=0$ and \(\mu(X)<\infty\).
Let $\phi\in\mathcal{I}$ with $\phi^{-1}(\R_+)\neq\emptyset.$ Then \(G\) satisfies
\SC{} if and only if one, and hence all, of the
statements {\rm(ii)}--{\rm(iv)} of
Theorem~\ref{thm:generalized-mass-characterization} hold, with the
generalized mass identity~\MB{} reducing
to the conservation of mass:
\[
    \mathsf M_u(t)=\mathsf M_{u_0}
    \qquad \text{for } t\in[0,T].
\]

Moreover, if \(G\) does not satisfy \SC{}, then, for every \(T>0\)
and every \(c \in \R_+\) such that $0<\phi(c)<\sup_{\R}\phi,$ there exist bounded positive pointwise solutions
\(u_-\) and \(u_+\), both with initial datum \(u_0= c\), and times
\(t_-,t_+\in(0,T]\) such that
\[
    0\le u_-\le c\le u_+
    \qquad\text{on }[0,T]\times X
\]
and
\[
    \mathsf M_{u_-}(t_-)<c\,\mu(X)<\mathsf M_{u_+}(t_+).
\]
Thus, failure of \SC{} allows for both loss of mass to infinity and
creation of mass from infinity.
\end{corollary}

\begin{proof}
Since \(\kappa=0\), stochastic completeness at infinity coincides
with stochastic completeness by definition, and the
compensating term in~\MB{} vanishes. The
stated equivalences are therefore the specialization of
Theorem~\ref{thm:generalized-mass-characterization} to the case
\(\kappa=0\).

For the two-sided assertion, suppose that \(G\) is stochastically
incomplete and fix \(T>0\) and \(c>0\). Since \(\kappa=0\), the
constant function \(c\) is a bounded pointwise solution and, for the
exterior value \(\alpha=c\), the finite subgraph
problems~\eqref{eq:killing-finite-alpha} with initial datum \(c\) are
solved by the constant \(c\), which is their unique solution by the
two-sided application of Lemma~\ref{lem:finite-comparison}:
\[
    u_n^c(t,x)= c.
\]
Fix now \(c\) such that $0<\phi(c)<\sup_{\R}\phi$ and choose \(\alpha>c\) such that $\phi(c)<\phi(\alpha)$.
The finite subgraph comparison principle,
applied to the ordered exterior data \(0<c<\alpha\), yields
\[
    0\le u_n^0\le u_n^c=c\le u_n^\alpha\le\alpha
    \qquad\text{on }[0,T]\times\Omega_n,
\]
and passing to the subsequential limits of the construction gives
bounded positive pointwise solutions \(u^0\) and \(u^\alpha\), both
with initial datum \(c\), satisfying
\[
    0\le u^0\le c\le u^\alpha\le\alpha
    \qquad\text{on }[0,T]\times X.
\]
Since \(\phi(0)\), \(\phi(c)\) and \(\phi(\alpha)\) are pairwise
distinct, the separation argument in Step~3 of the proof of
Theorem~\ref{thm:killing-nonlinear-nonunique} distinguishes the limit
solutions associated with distinct exterior values; in particular,
\[
    u^0\not= c
    \qquad \text{and} \qquad
    u^\alpha\not= c.
\]
As all three solutions share the initial datum \(c\), there exist
\(t_-,t_+\in(0,T]\) and \(x_-,x_+\in X\) with
\[
    u^0(t_-,x_-)<c <
    u^\alpha(t_+,x_+).
\]
Because the inequalities \(u^0\le c\le u^\alpha\) hold on all of \(X\)
and \(\mu>0\) at every vertex,
\[
    \mathsf M_{u^0}(t_-)<c\,\mu(X)<\mathsf M_{u^\alpha}(t_+)
\]
and the claim follows with \(u_-:=u^0\) and \(u_+:=u^\alpha\).
\end{proof}

The preceding theorem concerns pointwise solutions bounded on the entire
time interval.  The
following pointwise extension permits unbounded initial data when the
solution is \(\ell^1\)-continuous and bounded on every positive-time interval.

We note that for every $u_0\in\ell^{1,+}(X,\mu)$, the existence of a nonnegative $\ell^1$-mild solution is known; see~\cite[Theorem~4.1]{bianchi2026bounded} and~\cite{bianchi2022generalized}. Under the assumptions below, this canonical mild solution satisfies~\MB{} by~\cite[Theorem~5.5]{bianchi2026bounded}. The following proposition instead concerns any pointwise solution having the stated additional regularity; mild solutions are not automatically pointwise or bounded on positive-time intervals under our standing assumptions.
\begin{proposition}
\label{prop:unbounded-L1-pointwise-mass}
Assume that \(G\) satisfies \SCinf{} and
that either $\mu(X)<\infty$ or $\phi$ satisfies \LG{}.
Let \(u_0\in\ell^{1,+}(X,\mu)\) and
let \(u\) be a positive pointwise solution on \([0,T]\times X\) such
that
\[
    u\in C\bigl([0,T];\ell^1(X,\mu)\bigr),
    \qquad u(0)=u_0,
\]
and
\begin{equation}\label{eq:positive-time-boundedness}
    \sup_{(s,x)\in[\tau,T]\times X}u(s,x)<\infty
    \qquad\text{for every }\tau\in(0,T).
\end{equation}
Then, \(u\) satisfies~\MB{}. In
particular, $\mathsf M_u(t)$ satisfies the
differential form of the generalized mass balance, i.e., $\mathsf M_u \in AC([0,T])$
and $\mathsf M_u'(t)=-\sum_{x\in X}\kappa(x)\phi(u(t,x)).$
\end{proposition}

\begin{proof}
Fix \(0<\tau<t\le T\).  On \([\tau,t]\times X\), $u$ is
bounded by assumption and its initial value \(u(\tau,\cdot)\) lies in
\(\mathcal D_+\), so the proof of (i) implying (ii) in
Theorem~\ref{thm:generalized-mass-characterization}, applied on the
time-shifted interval, gives
\begin{equation}\label{eq:balance-from-tau}
    \mathsf M_u(t)
    +\int_\tau^t\sum_{x\in X}\kappa(x)\,\phi(u(s,x))\dd s
    =\mathsf M_u(\tau).
\end{equation}
As \(\tau\downarrow0\), the \(\ell^1\)-continuity of $u$ gives
\(\mathsf M_u(\tau)\to\mathsf M_{u_0}\), while monotone convergence
applied to the non-negative integral gives the integral from \(0\) to
\(t\).  This proves \MB{}.  The absolute continuity and
differential conclusions now follow from Lemma~\ref{lem:differential-mass-balance}.
\end{proof}

We now show the necessity of the additional assumptions of
$\mu(X)<\infty$ or \LG{} in the preceding results.

\begin{example}[Sharpness of the hypotheses in Theorem~\ref{thm:generalized-mass-characterization}]
\label{ex:sharpness-mass-extensions} \

\smallskip
\noindent\textbf{(a) The growth condition \LG{}
cannot  be dropped in the infinite measure case.}
Let \(X=\N_0\), \(\mu=1\), \(\kappa=0\), and set
\[
    s_n:=\sum_{k=0}^n\frac1{(k+1)^2}
    \qquad \text{and} \qquad
    w(n,n+1):=s_n(n+1)(n+2)
\]
with $w$ symmetric and all other edge weights vanishing.  The birth--death chain criterion for
stochastic completeness (see~\cite[Theorem~9.25]{KLW21}) gives
\[
    \sum_{n=0}^\infty
       \frac{\sum_{k=0}^n\mu(k)}{w(n,n+1)}
    =\sum_{n=0}^\infty\frac1{s_n(n+2)}
    =\infty,
\]
so this graph is stochastically complete.  Take the fast diffusion
power \(m=1/2\) and let
\[
    \phi(r)=r|r|^{-1/2}
    \qquad \text{and} \qquad
    v(n)=\frac1{(n+1)^2}.
\]
Then, \(v\in\ell^1(X,\mu)\cap\ellinfty\), but
\(\Phi v\notin\ell^1(X,\mu)\), and a direct edge computation gives
\[
    \Delta\Phi v=v.
\]
Consequently, letting
\[
    u(t,n)=\Bigl(1-\frac t2\Bigr)_+^{2}v(n)
\]
and noting that $\Delta \Phi u(t,n)=(1-t/2)_+ v(n)$, we obtain that $u$
is a pointwise solution with $u(t, \cdot) \in \mathcal{D}_+$ for all $t \geq 0$ and
\(u_0=v\in\mathcal D_+\), while
\[
    \mathsf M_u(t)
    =\Bigl(1-\frac t2\Bigr)_+^{2}\mathsf M_{u_0}
\]
is not conserved so that \MB{} fails for this example.

We note that \LG{} does not hold.  Thus, on arbitrary measure
graphs, some control of \(\phi(r)/r\) near zero is indispensable for a
statement covering all bounded finite-mass data.

\smallskip
\noindent\textbf{(b) Positive-time boundedness cannot be dropped, even in
finite measure.}
Again take \(X=\N_0\) and \(\kappa=0\), now with
\[
    \mu(n)=2^{-n},
    \qquad
    v(n)=n+1
    \qquad \text{and} \qquad
    \phi(r)=r|r|
\]
and define
\[
    q_n:=\sum_{k=0}^n2^{-k}(k+1)
    \qquad \text{and} \qquad
    w(n,n+1):=\frac{q_n}{2n+3}.
\]
Then, \(\mu(X)=2\), \(q_n\uparrow4\), and the birth--death criterion
again gives stochastic completeness, since
\[
    \sum_{n=0}^\infty
      \frac{\sum_{k=0}^n2^{-k}}{w(n,n+1)}=\infty.
\]
Moreover, \(v\in\ell^1(X,\mu)\) and the identities
\[
    w(n,n+1)\bigl((n+1)^2-(n+2)^2\bigr)=-q_n
    \qquad \text{and} \qquad
    w(n-1,n)\bigl((n+1)^2-n^2\bigr)=q_{n-1}
    \quad \text{for }n\ge1 ,
\]
show that
\[
    \Delta\Phi v=-v.
\]
Hence,
for every \(T<1\), letting
\[
    u(t,n)=\frac{n+1}{1-t}
    \qquad \text{for }0\le t\le T
\]
and noting that $\Delta \Phi u = -\left({1-t}\right)^{-2}v$, we get that $u$
belongs to \(C^1\bigl([0,T];\ell^1(X,\mu)\bigr)\), satisfies the
equation pointwise and in \(\ell^1(X,\mu)\), but creates mass:
\[
    \mathsf M_u(t)=\frac4{1-t}>4=\mathsf M_{u_0}
    \qquad \text{for }t>0 .
\]
Here, the nonlinearity satisfies \LG{} and \(u\) is unbounded at
every time, so neither
Theorem~\ref{thm:generalized-mass-characterization} nor
Proposition~\ref{prop:unbounded-L1-pointwise-mass} applies; being
unbounded, \(u\) also lies outside the class of
Theorem~\ref{thm:killing-unique}, so there is no conflict with
uniqueness.  Thus, the generalized mass balance~\MB{} cannot be extended to
arbitrary unbounded \(\ell^1\)-valued pointwise solutions, even in
finite measure.
\end{example}

We conclude with a discrete counterpart of the measure-change
characterization of Masamune--Schmidt
\cite[Theorem~3.5(b)]{masamune2020generalized}.
In the graph setting, the equivalence below follows immediately from the bounded injectivity
characterization of stochastic completeness at infinity given in
Proposition~\ref{prop:SCinf-characterizations}~(ii).

This result shows that stochastic completeness at
infinity is ordinary stochastic completeness after absorbing the
killing term into the measure.
In particular, every criterion for the stochastic completeness of
killing-free graphs, such as those in terms of volume growth, see \cite{Fol14, HKS20},
or curvature, see \cite{HL17, MW19},
yields, via the passage from \(\widetilde G\) to $G$ in the proposition below,
a criterion for the generalized conservation of mass.

\begin{proposition}
\label{prop:time-change-SCI}
Let \(G=(X,w,\kappa,\mu)\) be a graph. Define
\[
    \widetilde\mu:=\mu+\kappa
    \qquad \text{and} \qquad
    \widetilde G:=(X,w,0,\widetilde\mu).
\]
Then,
\[
    G\text{ satisfies }\SCinf{}
    \quad\Longleftrightarrow\quad
    \widetilde G\text{ satisfies }\SC{}.
\]
Moreover, if \((X,w,0,\mu)\) satisfies \SC{}, then
\((X,w,\kappa,\mu)\) satisfies \SCinf{} for every
killing term \(\kappa\ge0\).
\end{proposition}

\begin{proof}
Since \(w\) is unchanged and \(\widetilde\mu>0\),
\(\widetilde G\)  is a graph and bounded functions lie in the domain of
both formal Laplacians. For every \(h\in\ellinfty\) and \(x\in X\),
multiplying through by the respective measures,
\begin{align*}
    \bigl(1+\Delta_{\mu,\kappa}\bigr)h(x)=0
    &\Longleftrightarrow
    \sum_{y\in X}w(x,y)\bigl(h(x)-h(y)\bigr)
    +\bigl(\mu(x)+\kappa(x)\bigr)h(x)=0\\
    &\Longleftrightarrow
    \bigl(1+\Delta_{\widetilde\mu,0}\bigr)h(x)=0.
\end{align*}
By \SCinf{} as characterized in Proposition~\ref{prop:SCinf-characterizations}~(ii)
with \(\lambda=1\), injectivity of
\(1+\Delta_{\mu,\kappa}\) on \(\ellinfty\) is \SCinf{} for \(G\), while
injectivity of \(1+\Delta_{\widetilde\mu,0}\) on \(\ellinfty\) is
\SCinf{} for \(\widetilde G\), which, in the absence of killing, is
\SC{} by definition. This proves the
first assertion.

For the second assertion, see \cite[Remark~(d) following Theorem~1]{KL12}.
\end{proof}

\section*{Acknowledgments}
D.~B.~is supported by the Startup Fund of Sun Yat-sen University.
B.~H.~is supported by NSFC, no.~12371056.
A.~G.~S.~is a member of the GNAMPA-INdAM group {``Equazioni Differenziali e Sistemi Dinamici''}.
R.~K.~W.~is supported by the Simons Foundation, in the form
of a Travel Support for Mathematicians gift, and was supported by the PSC-CUNY, in the form
of a Department Chair CUNY Research Foundation (RF) Account.

\printbibliography

\end{document}